\documentclass{amsart}
 
\usepackage{amscd}
\usepackage{amsmath}
\usepackage{amssymb}
\usepackage{amsthm}
\usepackage{pinlabel}
\usepackage{epsfig}
\usepackage[all,cmtip]{xy}
\usepackage{graphicx,color}
\usepackage{subfigure}
\usepackage{lineno}

\theoremstyle{plain}
\newtheorem{thm}{Theorem}[section]
\newtheorem{lem}[thm]{Lemma}
\newtheorem{prop}[thm]{Proposition}

\newtheorem{rem}[thm]{Remark}

\newcommand{\Sth}{S^3}
\newcommand{\Rth}{\mathbb{R}^3}
\newcommand{\Hth}{\mathbb{H}^3}

\newcommand{\GM}{\Gamma_M}
\newcommand{\GK}{\Gamma_K}
\newcommand{\GQ}{\Gamma_Q}

\newcommand{\PSLTC}{PSL(2, \mathbb{C})}

\newcommand{\OD}{\mathbb{O}_d}
\newcommand{\PSLTOD}{PSL(2, \OD)}

\newcommand{\PSLTA}{PSL(2, \mathbb{A})}
\newcommand{\SOTR}{SO(3, \mathbb{R})}

\newcommand{\Z}{\mathbb{Z}}
\newcommand{\Q}{\mathbb{Q}}

\def\co{\colon\thinspace}

\begin{document}

\title{Small knot complements, exceptional surgeries, and hidden symmetries}

\author{Neil Hoffman}
\address{Department of Mathematics\\
Boston College\\ 
Chestnut Hill, MA 02467}
\email{neil.hoffman@bc.edu}

\date{\today}

\begin{abstract}
This paper provides two obstructions to small knot complements in $\Sth$ admitting hidden symmetries. The first obstruction is being cyclically commensurable with another knot complement. This result provides a partial answer to a conjecture of Boileau, Boyer, Cebanu, and Walsh. We also provide a second obstruction to admitting hidden symmetries in the case where a small knot complement covers a manifold admitting some 
symmetry and at least two exceptional surgeries. 
 \end{abstract}

\maketitle
\section{Introduction}\label{sect:intro}

Two finite volume, orientable, hyperbolic 3-orbifolds are \emph{commensurable} if they share a common finite sheeted cover. 
Using Mostow-Prasad rigidity (see \cite{Mostow}, \cite{Prasad}), we may restate this definition in terms of Kleinian groups with finite co-volume. Namely, $\Gamma_1$ and $\Gamma_2$ are \emph{commensurable} if there exists $g \in\PSLTC$ such that $\Gamma_1$ and $g\Gamma_2 g^{-1}$ share a common finite index subgroup.   
Commensurability defines an equivalence relation both on the set of hyperbolic 3-orbifolds and on their associated Kleinian groups. 
One way to distinguish these equivalence classes, called commensurability classes, is to identify elements that are rare in a given commensurability class. Conjecturally, there are at most three hyperbolic knot complements in a given commensurability class (see \cite[Conj 5.2]{RW}).  Both here and throughout the paper, we will use the term \emph{knot complement} to refer to a space homeomorphic to $\Sth-K$ where $K$ is a smoothly embedded knot in $\Sth$.

In \cite{BBCW}, the conjecture is shown to hold if the knot complements do not admit \emph{hidden symmetries}, ie the knot complements do not cover an orbifold with a rigid cusp (see $\S$\ref{sect:background}). Furthermore, two commensurable knot complements either admit hidden symmetries, or are \emph{cyclically commensurable}, ie they both cyclically cover the same orbifold (see \cite[Thm 1.4]{BBCW}). For a \emph{small} knot complement, ie the knot complement does not contain a closed, embedded, essential surface, the theorem below shows these are mutually exclusive, which is a partial solution to Conjecture 4.14 of \cite{BBCW}.

\begin{thm}\label{mainthm1}
Let $\Sth-K$ be a small hyperbolic knot complement. 
If $\Sth-K$ admits hidden symmetries, it is not cyclically commensurable with another 
knot complement.
\end{thm}

Typically, obstructions to covering rigid cusped orbifolds arise from the computation of 
arithmetic invariants, eg the invariant trace field (see $\S$\ref{sect:background}). 
The following theorem provides an obstruction to 
covering a rigid cusped orbifold by combining more topological and geometric criterion namely smallness, exceptional surgeries, and in the $S^2(2,3,6)$ case, symmetry. 

\begin{thm}\label{mainthm2}
Let $M$ be a manifold covered by a small hyperbolic knot complement $\Sth-K$
that is not the figure 8 knot complement. 
\begin{enumerate}
\item If $M$ admits two exceptional
surgeries, then $M$ does not cover an orbifold with a $S^2(2,4,4)$ or $S^2(3,3,3)$ cusp. 
\item If $M$ admits two exceptional
surgeries and a non-trivial symmetry, then $M$ does not cover an orbifold with a $S^2(2,3,6)$ cusp. 
\end{enumerate}
\end{thm}

This paper is organized as follows. In $\S$\ref{sect:background}, we provide the necessary background including a description of many of the smallest volume cusped orbifolds. We then establish two lemmas that together characterize the representations of the fundamental groups of small knot complements in $\PSLTC$ in $\S$\ref{sect:knotReps}. In $\S$\ref{sect:Thm1Proof}, we prove Theorem \ref{mainthm1}. And in the following section, we prove Theorem \ref{mainthm2}. In $\S$\ref{sect:FurtherRemarks}, we provide some final remarks.

\subsection{Acknowledgements}
This work started under the direction of my PhD advisor Alan Reid, who patiently and thoughtfully provided feedback on this work at every stage of development. He deserves many thanks for this. Also, I would like to thank
Colin Adams, Ken Baker, Jason Deblois,  Cameron Gordon, Brandy Guntel, Keenan Kidwell,  John Luecke, Matthew Stover, and Genevieve Walsh for a number of helpful conversations. 

\section{Background}\label{sect:background}

We begin by establishing some notation. 
First, we consider all groups as subgroups of $\PSLTC$ with finite covolume 
unless explicitly stated otherwise. And therefore, we consider all manifolds and orbifolds to be both hyperbolic and orientable unless explicitly stated otherwise. 
We denote by $\GQ = \pi_1^{orb}(Q)$. Hence, we may say $Q=\Hth/\GQ$. In the case of knot complements, we say that $\GK=\pi_1(\Sth-K)$.
Furthermore, it will prove convenient to discuss the subgroup that fixes $\infty$, we call such a subgroup the \emph{peripheral subgroup} of $\GQ$ (or $\GK$) and denote it by $P_Q$ (or $P_K$). 

A distinct advantage of considering groups as discrete subgroups of $\PSLTC$ is that they carry a considerable amount of number theoretic information. In order to fully utilize this data, we observe the following standard notation. First, we denote by $\mathbb{A}$ is the set of algebraic integers in $\mathbb{C}$.
If $L$ is a number field, we will use $\mathcal{O}_L$ to denote the ring of integers in $L$ and $\alpha\mathcal{O}_L$ to denote the principal ideal generated by $\alpha\in\mathcal{O}_L$. If $I$ is a non-principal ideal in $\mathcal{O}_L$, we will denote it by $<\alpha, \beta>$
where I is generated by 
 the $\mathcal{O}_L$ linear combinations of $\alpha$ and $\beta$.

For a group $\Gamma$, $g\in\Gamma$ is a two-element coset of the form 
$$\left\{ \begin{pmatrix}a & b\\ c & d\end{pmatrix}, \begin{pmatrix} -a &-b\\ -c & -d\end{pmatrix} \right\}.$$ 
Later, we will abuse notation and just refer to an element by one of the above matrices.
We say the trace of $g$ or $tr(g)=|a+d|$. If for all $g\in \Gamma$, $tr(g) \in \mathbb{A}$, we say $\Gamma$ has \emph{integral traces}. Otherwise, $\Gamma$ has \emph{non-integral traces}. In addition, if $\Gamma$ integral traces then $\Gamma$ admits a representation into $\PSLTA$, a fact exploited in Lemma \ref{lem:intTracesUpperTri}.

Next,  the \emph{invariant trace field} of $\Gamma$ or $k\Gamma$ to be
$\Q(\{tr(g^2)|g\in\Gamma\})$. Reid showed that this is an invariant of the commensurability class of $\Gamma$ (see \cite{Reid2}). 

A cusped 3-orbifold $Q=\Hth/\Gamma$ (or $\Gamma$) is \emph{arithmetic} if $\Gamma$ is commensurable with $\PSLTOD$ where
$\OD$ is the ring of integers in $\Q(\sqrt{-d})$ and $d$ is a positive square-free integer.  Otherwise, $Q$ (or $\Gamma$) is \emph{non-arithmetic}.

\subsection{Rigid Cusps}\label{sect:rigidCusps}
The cusps of an orientable 3-orbifold are one of five types (see \cite[$\S$3]{Walsh2011} for background). The \emph{non-rigid} Euclidean 2-orbifolds are the torus and $S^2(2,2,2,2)$. The \emph{rigid} orbifolds are  $S^2(2,4,4)$, $S^2(2,3,6)$, and $S^2(3,3,3)$.  

It is worth mentioning that the definition of hidden symmetries in the introduction is slightly atypical. Traditionally, a group $\Gamma$ admits hidden symmetries if $[Comm(\Gamma):N(\Gamma)]>1$, where 
$$Comm(\Gamma)=\{g\in\PSLTC| [g\Gamma g^{-1}:g\Gamma g^{-1}\cap \Gamma]< \infty \mbox{ and }  [\Gamma:g\Gamma g^{-1}\cap \Gamma]< \infty\}$$
and $N(\Gamma)$ is the normalizer of $\Gamma$ in $\PSLTC$. We say $Q$ admits hidden symmetries if $\Gamma_Q$ admits hidden symmetries. 
We say $Comm(\Gamma)$ is the (orientable) commensurator. If $\Gamma$ is non-arithmetic, then $Comm(\Gamma)$ is discrete and we can define $\Hth/Comm(\Gamma)$ as the (orientable) commensurator quotient (see \cite{Margulis1991}). By construction, for non-arithmetic $\Gamma$,  $\Hth/Comm(\Gamma)$
is covered by all orbifolds its commensurability class and therefore, it is the smallest volume
orbifold in the commensurability class of $\Hth/\Gamma$. 
Also, for a non-arithmetic knot complement, having hidden symmetries is equivalent to $\Hth/Comm(\GK)$ having a rigid cusp (see \cite[Prop 9.1]{NR1}). 

In arguments that follow, we will need to exploit facts about many of the smallest volume cusped orbifolds, all of which have rigid cusps. 
This accounting of small volume orbifolds heavily relies on Meyerhoff's result that the densest horoball packing has a cusp density of $\frac{\sqrt{3}}{2 v_0}$, where \emph{cusp density} is defined to be cusp volume divided by total volume and $v_0 \approx 1.0149416$ is the volume the regular ideal tetrahedron (see \cite{Meyerhoff1985}).
 The following theorem summarizes Meyerhoff's result and  Adams' classification of small cusp volume hyperbolic orbifolds  
(see \cite[Thm 3.2, Cor 4.1, Thm 5.2]{Adams1}).

\begin{thm}[Adams, 1991]\label{thm:AdamsCusps}

Let $Q$ be a 1-cusped hyperbolic 3-orbifold.
\begin{enumerate}
\item A maximal $S^2(2,3,6)$ cusp in $Q$ has volume either $\frac{\sqrt{3}}{24}$, $\frac{\sqrt{3}}{12}$, $\frac{1}{8}$,
$\frac{\sqrt{3}(3+\sqrt{5})}{48}$, $\frac{\sqrt{21}}{24}$ or at least $\frac{\sqrt{3}}{8}$.
\item A maximal $S^2(3,3,3)$ cusp in $Q$ has volume either $\frac{\sqrt{3}}{12}$, $\frac{\sqrt{3}}{6}$, $\frac{1}{4}$,
$\frac{\sqrt{3}(3+\sqrt{5})}{24}$, $\frac{\sqrt{21}}{12}$ or at least $\frac{\sqrt{3}}{4}$.
\item A maximal $S^2(2,4,4)$ cusp in $Q$ has volume either $\frac{1}{8}$, $\frac{\sqrt{2}}{8}$, or at least $\frac{1}{4}$.
\end{enumerate}
\end{thm}

Adams points out that for each a cusp volume explicitly listed in 1), there is a unique orbifold $Q$  with a $S^2(2,3,6)$ cusp. Moreover, each of these orbifolds has a unique double cover with a $S^2(3,3,3)$ cusp corresponding to an orbifold with cusp volume explicitly listed in 2).

Neumann and Reid provided detailed descriptions of many of the orbifolds corresponding to the cusp volumes in these theorems (see \cite{NR2}).
A number of the orbifolds they describe are arithmetic. More specifically, their notes together with some notes by Adams on the volume of the orbifolds
can be summarized in the following proposition (see \cite{Adams1}, \cite{NR2}):

\begin{prop}[Adams 1991, Neumann and Reid 1991]\label{prop:AdamsAndNR2}
Let $Q$ be a 1-cusped hyperbolic 3-orbifold.

\begin{enumerate}

\item If $Q$ has a maximal $S^2(2,3,6)$ cusp of volume either $\frac{\sqrt{3}}{24}$, $\frac{\sqrt{3}}{12}$, or $\frac{1}{8}$,
it is arithmetic. Furthermore, these orbifolds have volumes $\frac{v_0}{12}$, $\frac{v_0}{6}$, and $\frac{5v_0}{24}$, respectively and no
other orbifold with this cusp type is of lower volume.
\item If $Q$ has a maximal $S^2(3,3,3)$ cusp of volume either $\frac{\sqrt{3}}{12}$, $\frac{\sqrt{3}}{6}$, or $\frac{1}{4}$,
it is arithmetic. Furthermore, these orbifolds have volumes $\frac{v_0}{6}$, $\frac{v_0}{3}$, and $\frac{5v_0}{12}$, respectively and no
other orbifold with this cusp type is of lower volume. 
\item If $Q$ has a maximal $S^2(2,4,4)$ cusp of volume either $\frac{1}{8}$ or $\frac{\sqrt{2}}{8}$, 
it is arithmetic. Furthermore, these orbifolds have volumes $\frac{v_1}{6}$ and $\frac{v_1}{4}$, respectively and no
other orbifold with this cusp type is of lower volume.
\end{enumerate}
\end{prop}

We use the standard notation that $v_1 \approx .91596244$ is the volume of the ideal tetrahedron in $\Hth$ with dihedral angles of $\frac{\pi}{2}$, 
$\frac{\pi}{4}$, and $\frac{\pi}{4}$.

Additionally, the orbifold with a $S^2(2,3,6)$ cusp and cusp volume $\frac{\sqrt{3}(3+\sqrt{5})}{48}$ is the tetrahedral orbifold $\Hth/\Gamma(5,2,2;2,3,6)$ with volume approximately 0.343003 (see \cite[$\S$9]{NR1} and \cite[p. 144]{MR} for background). As noted by Adams, this orbifold has a unique 2-fold cover with a $S^2(3,3,3)$ cusp and cusp volume $\frac{\sqrt{3}(3+\sqrt{5})}{24}$. This orbifold is $\Hth/\Gamma(5,2,2;3,3,3)$  (see \cite{Adams1}).

We delay discussion of orbifolds of other relevant cusp volumes until  $\S$\ref{sect:Thm2Proof}.

\subsection{The isotropy graph and finite subgroups of $SO(3,\mathbb{R})$ }\label{sect:SubGroupsofSOn}

This subsection draws upon Thurston's definition of a (geometric) 3-orbifold. For further background, we refer the reader to
\cite[Chapter 13]{Thurston1} and \cite{Walsh2011}.

We define the \emph{base space} of an orbifold $Q$ to be the underlying topological space. For convenience, we use $|Q|$ to denote the base space of an orbifold $Q$.
In dimensions 2 and 3, $|\mathbb{R}^2/\langle \gamma | \gamma^n \rangle|=\mathbb{R}^2$ and $|\Rth/G|=\Rth$. Hence, the base space of a
2-orbifold is a surface and the base space of a 3-orbifold is a 3-manifold. 
Also, if all neighborhoods of $x\in Q$ map to $\mathbb{R}^2/\langle \gamma | \gamma^n \rangle$ or $\Rth/G$, we call 
x a \emph{cone point} of $Q$.

An \emph{elliptic} 2-orbifold is an orientable 2-orbifold that can be covered by $\mathbb{S}^2$. 
The complete list of orientable 2-orbifolds covered by $\mathbb{S}^2$ is $\mathbb{S}^2$, $\mathbb{S}^2(n,n)$, $\mathbb{S}^2(2,2,n)$, $\mathbb{S}^2(2,3,3)$, $\mathbb{S}^2(2,3,4)$ and
$\mathbb{S}^2(2,3,5)$. 

Taking the cone over each of these orbifolds produces all of the possibilities for $\Rth/G$ (see Fig \ref{fig:FiveTypesOfIntCones}). In particular, $G$ is either trivial, finite cyclic, 
$D_n$ (a dihedral group of order $2n$),
$A_4$, $S_4$, or $A_5$. Finally, if $x\in \Hth/\Gamma$ the \emph{isotropy group} of $x$ is $G\subset\Gamma$ such
$g\in G$ if and only if $g(x)=x$. The set of all points fixed by some element in a fundamental domain for $Q$ is a trivalent graph, which we will refer
to as the \emph{isotropy graph}. We note that this graph need not be connected.  However, we may consider an orbifold as a base space together with an embedded isotropy graph.

\begin{figure}
\begin{center}
\resizebox{4.5 in}{!}{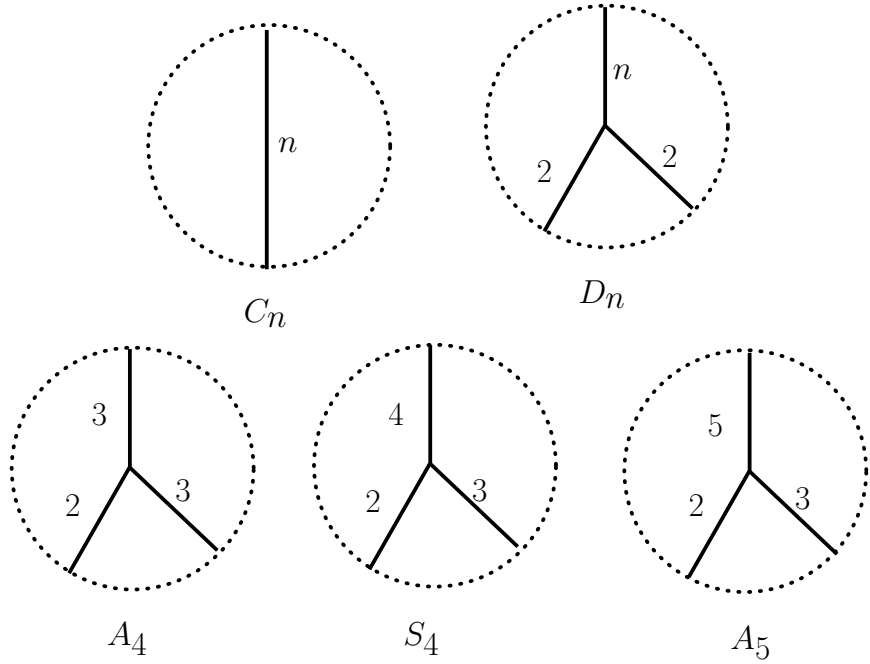}
\caption{\label{fig:FiveTypesOfIntCones} The five types of trivalent points that correspond to finite subgroups of $\SOTR$}
\end{center}
\end{figure} 
 
The specific embedding of the isotropy graph will be useful in the arguments of this paper. We will pay particular interest as to the isometric embeddings of the vertices of the isotropy graph. Using the spherical cosine law (see \cite[Lem 3]{Roeder2006} for an application in this context), 
 we can compute the angles between the axes for each type of isotropy group in Figure \ref{fig:FiveTypesOfIntCones}. If we consider
 the three axes fixed by elements of torsion of orders $a$, $b$, and $c$  where the angles between these axes are $\alpha$, $\beta$, $\gamma$
 (see in Fig \ref{fig:SO3Angles}), then

\begin{center} 
$\begin{array}{c}
\cos \alpha = \frac{\cos \frac{\pi}{a} + \cos \frac{\pi}{b} \cos \frac{\pi}{c}  }{\sin \frac{\pi}{b} \sin\frac{\pi}{c} }\\
 
 \\
\cos \beta = \frac{\cos \frac{\pi}{b} + \cos \frac{\pi}{a} \cos \frac{\pi}{c}  }{\sin \frac{\pi}{a}\sin\frac{\pi}{c}}\\
\\
 
\cos \gamma = \frac{\cos \frac{\pi}{c} + \cos \frac{\pi}{a} \cos \frac{\pi}{b}  }{\sin \frac{\pi}{a}\sin\frac{\pi}{b}}.
\end{array}$
 \end{center}
 
 In particular, if $(a,b,c)$ is $(2,2,n)$, then $(\alpha,\beta,\gamma)$ is $(\frac{\pi}{2},\frac{\pi}{2},\frac{\pi}{n})$.
If $(a,b,c)=(2,3,3)$, then $(\alpha,\beta,\gamma)$ is $(\cos^{-1}(\frac{1}{3}),\cos^{-1}(\frac{1}{\sqrt{3}}),\cos^{-1}(\frac{1}{\sqrt{3}}))$.
If $(a,b,c)=(2,3,4)$, then $(\alpha,\beta,\gamma)=(\cos^{-1}(\frac{1}{\sqrt{3}}),\frac{\pi}{4},\cos^{-1}(\frac{\sqrt{2}}{\sqrt{3}}))$.
If $(a,b,c)=(2,3,5)$, then $(\alpha,\beta,\gamma)=(\cos^{-1}(\frac{\cos(\frac{\pi}{5})}{\sqrt{3}\sin(\frac{\pi}{5})}),\cos^{-1}(\frac{1}{2\sin(\frac{\pi}{5})}),\cos^{-1}(\frac{2\cos(\frac{\pi}{5})}{\sqrt{3}}))$.
We note that only the dihedral isotropy groups have the property that there is an axis perpendicular to all other fixed point axes. 

\begin{figure}
\begin{center}
\resizebox{1.5 in}{!}{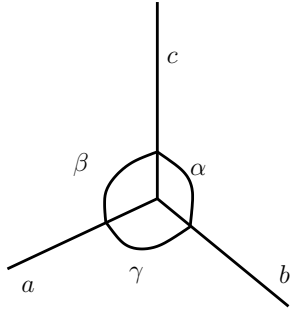}
\caption{\label{fig:SO3Angles} The angles between axes of fixed points in finite subgroups of $\SOTR$}
\end{center}
\end{figure}

\subsection{The cusp killing homomorphism}

Let $Q=\Hth/\Gamma_Q$ be a 1-cusped hyperbolic 3-orbifold.
Denote by $|\gamma|$ the order of an element in $\Gamma_Q$ and denote by
 $R = \{  \gamma | \gamma \in P_Q, |\gamma|<\infty\}$. 
 We define the 
 \emph{cusp killing homomorphism} to be
  $$f\co \Gamma_Q \rightarrow \Gamma_Q/\langle\langle R \rangle\rangle_{\Gamma_Q}.$$

We will make use of the following proposition, which was also noticed by M. Kapovich. 

\begin{prop}\label{cusp_kill}
Let $\Sth-K$ be a hyperbolic knot complement. 
Suppose $\Sth-K$ covers an orientable orbifold $Q$ with a non-torus
cusp.  Denote the cusp killing homomorphism by $f$.  
Then, $f(\Gamma_Q)$ is trivial. Furthermore, $|Q|\cong D^3$ and each component of
the isotropy graph of Q is connected to the cusp.
\end{prop}

\begin{proof}
 First note that $\Gamma_Q = P_Q \cdot \GK$. 

Since a meridian $\mu$ of $\GK$ is contained in $P_Q$ and $P_Q$ is generated by torsion elements on the cusp (we recall
$\S$\ref{sect:rigidCusps})
killing these torsion elements kills $\langle \langle \mu \rangle \rangle_{\Gamma_Q}$ as well as killing $P_Q$. However, 
$\GK =  \langle \langle \mu \rangle \rangle_{\GK}$ and $ \langle \langle \mu \rangle \rangle_{\GK} \subset  \langle \langle \mu \rangle \rangle_{\Gamma_Q}$.
Hence, the cusp killing homomorphism kills the whole group $\Gamma_Q$.

Thus, $|Q|$ is a simply connected space with $S^2$ boundary. Therefore, $|Q|\cong D^3$ by the solution to the Poincar\'{e} Conjecture 
(see \cite{MT}). 

If there were any pieces of the isotropy graph not connected to the cusp, 
then there would be elements of finite order that are non-trivial under the cusp killing homomorphism. Hence, the isotropy graph is connected.
\end{proof}

\begin{figure}
\centering
\subfigure[ The isotropy graph before reduction
]{
\resizebox{6 cm}{!}{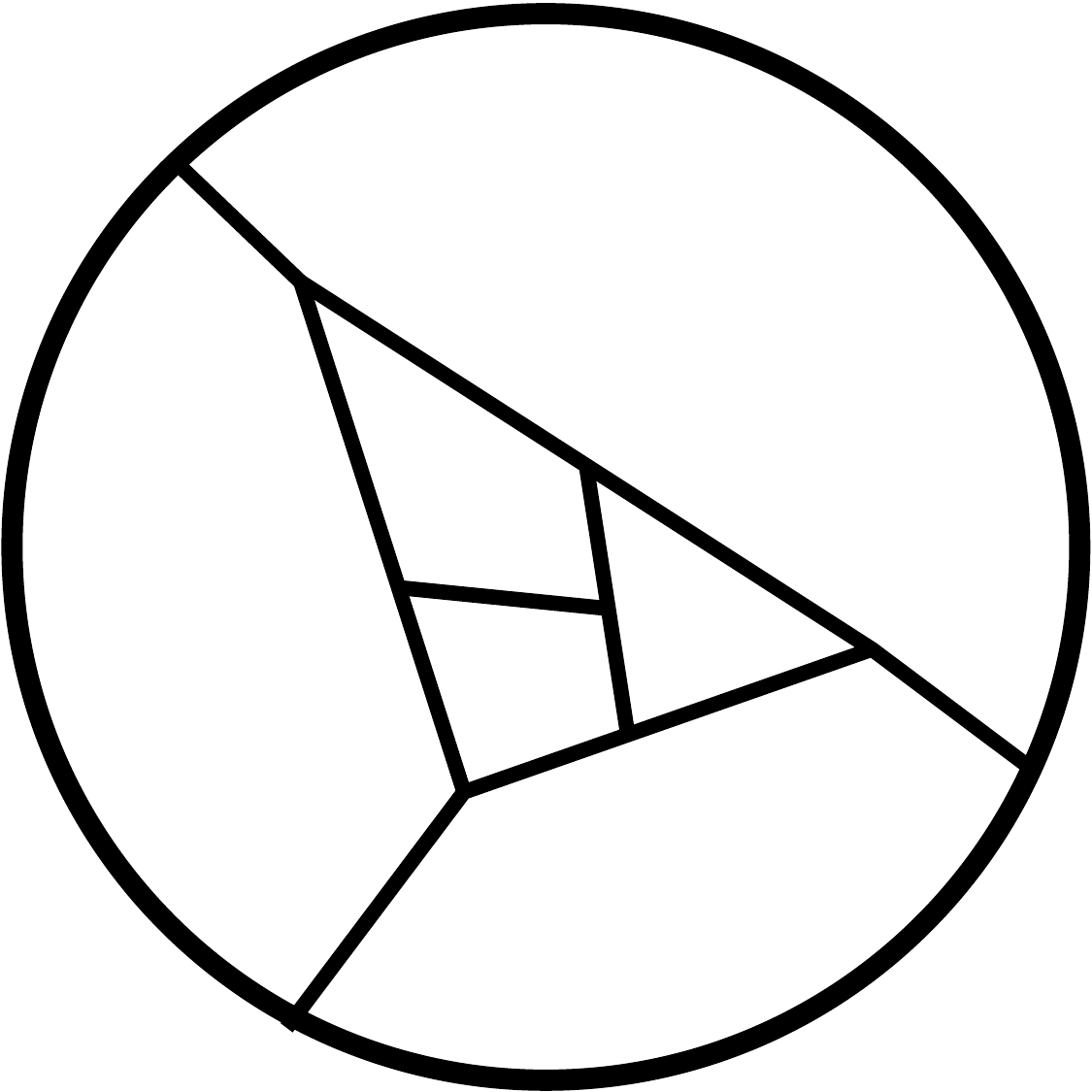}
\label{fig:GraphCuspKillStep1}
}
\subfigure[The graph after removing edges corresponding to torsion elements fixing points on the cusp]{
 \resizebox{6 cm}{!}{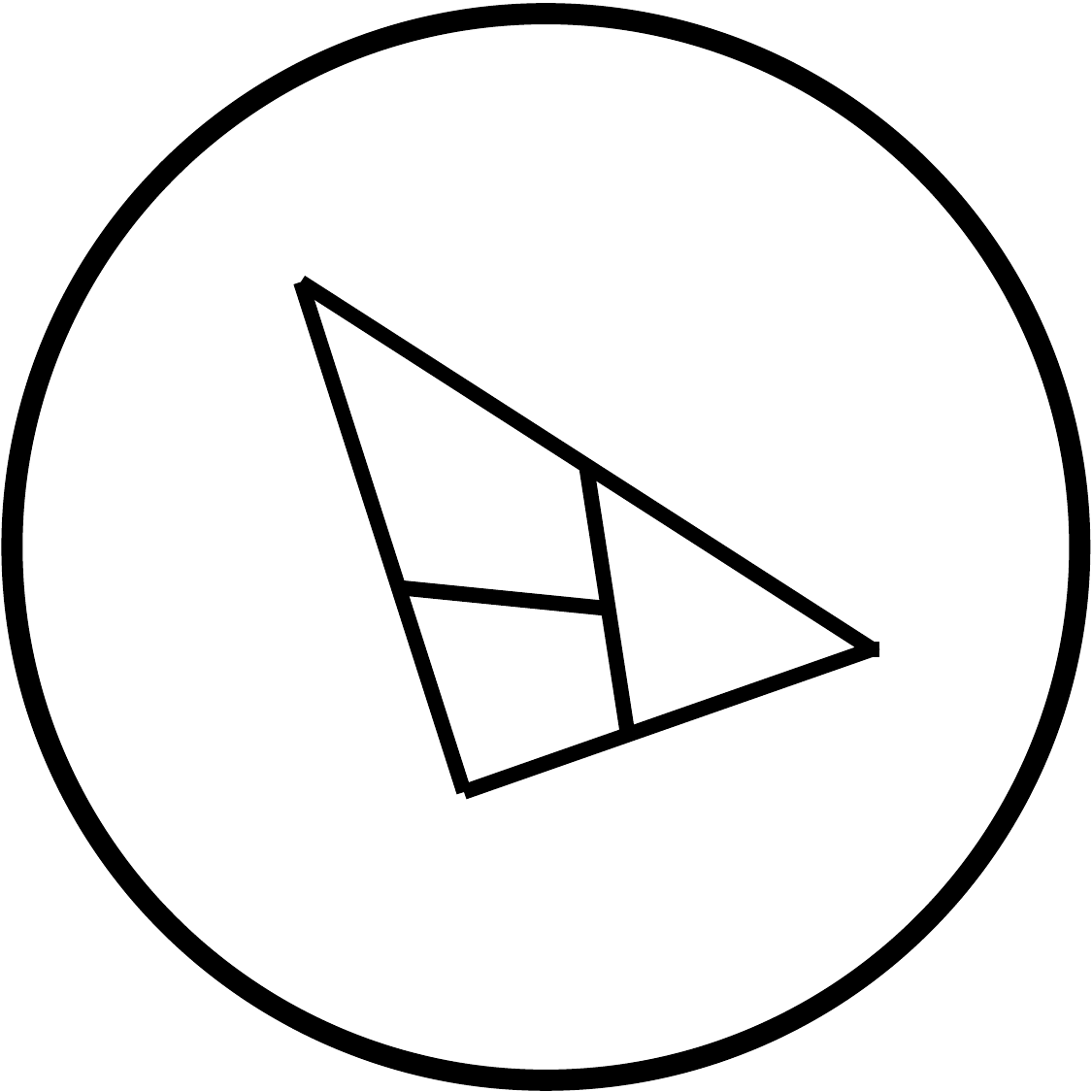}  
}
\subfigure[The graph after resolving the degree 2 vertices
]{
\resizebox{6cm}{!}{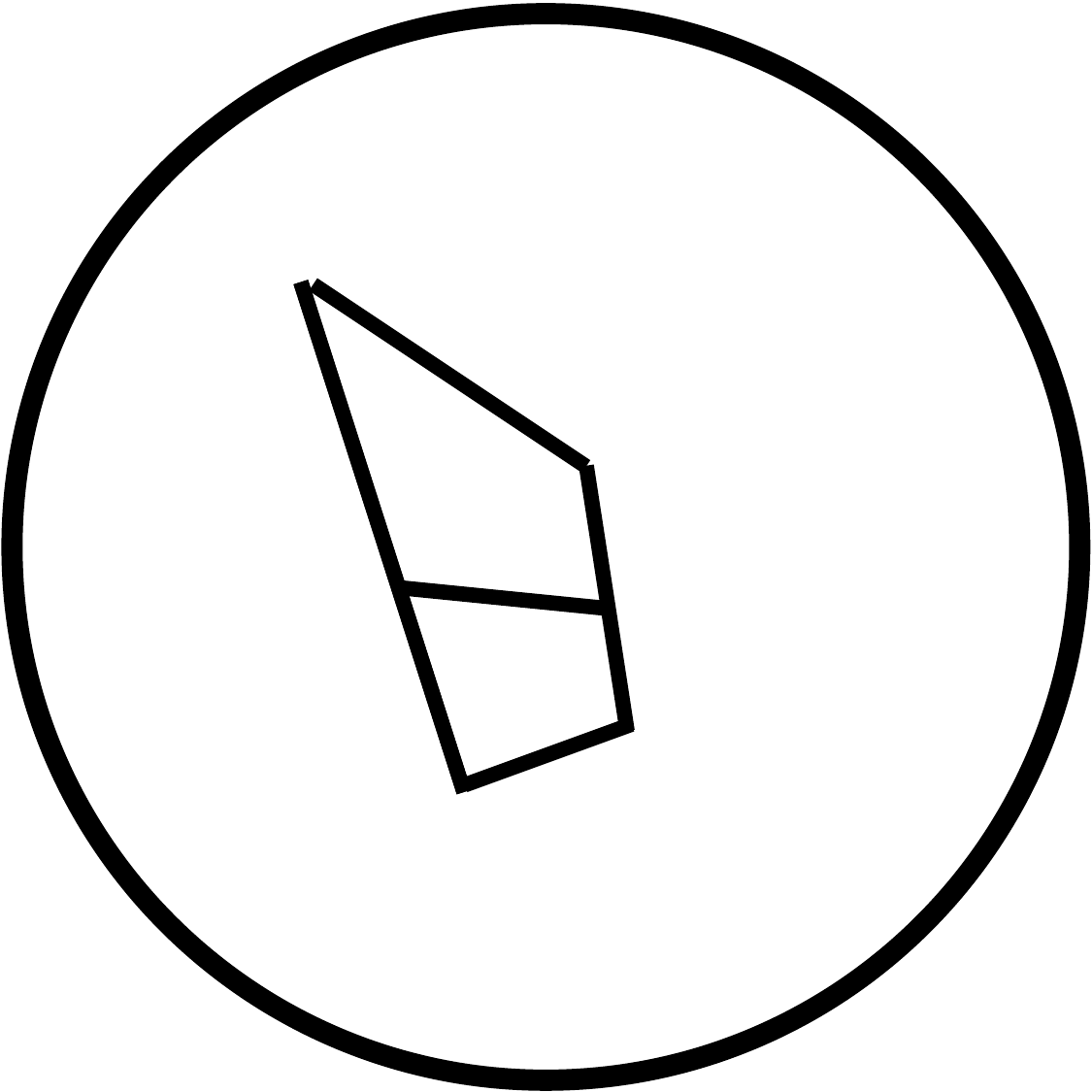} 
\label{fig:GraphCuspKillStep3}
}
\subfigure[The result of cusp killing is a graph corresponding to a dihedral group]{
\resizebox{6 cm}{!}{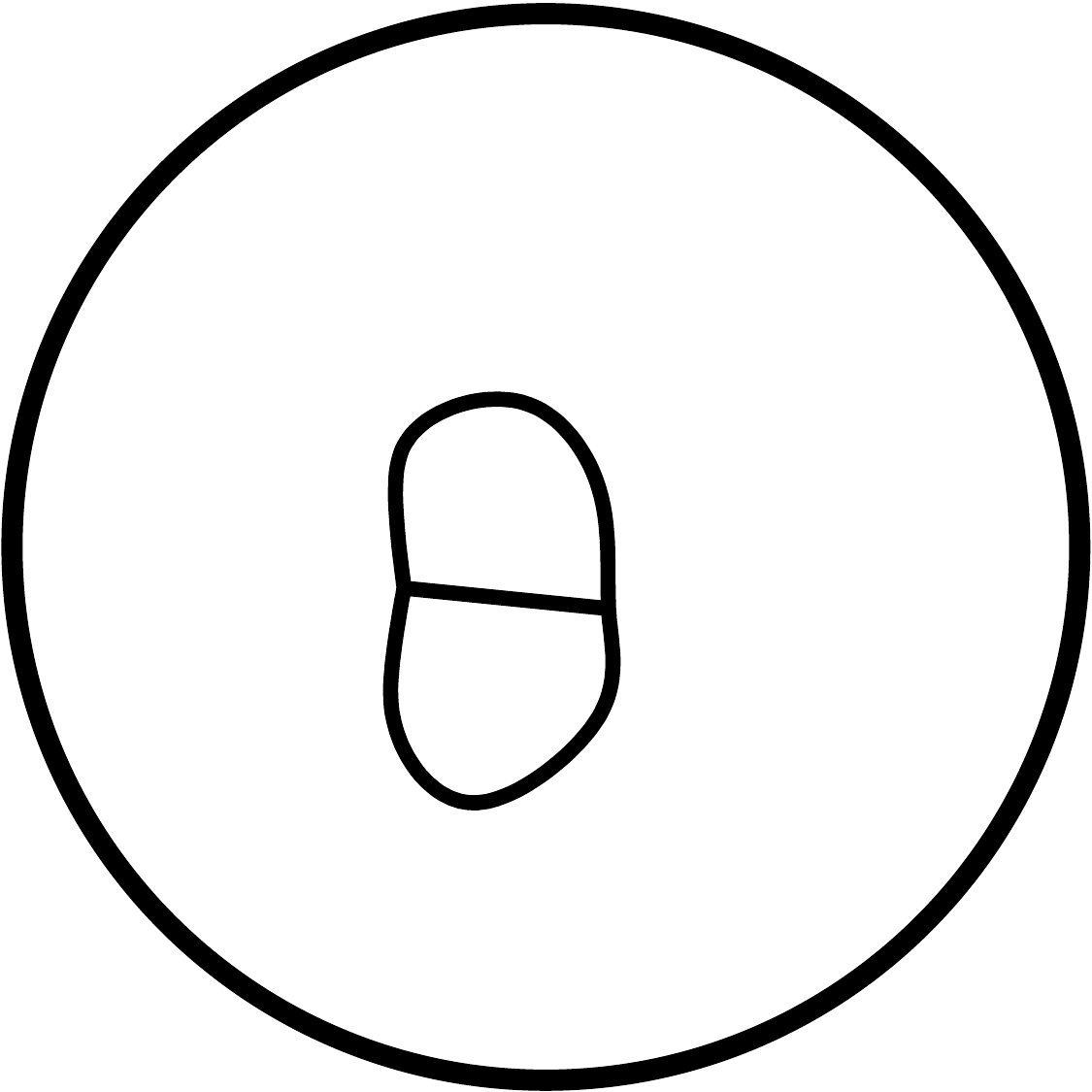}
\label{fig:GraphCuspKillStep4}
}

\caption{\label{fig:GraphCuspKill} A step by step graphical interpretation of the cusp killing homomorphism}
\end{figure}

\begin{rem}
We can also interpret the effects of the cusp killing homomorphism on the 
isotropy graph of $Q$ when $|Q|$ is simply connected. Viewing the isotopy graph as a weighted graph that generates the fundamental group of 
$Q$ via the Wirtinger presentation, 
killing elements of torsion on the cusp corresponds to erasing edges of the graph.
For each endpoint $x$ of an erased edge corresponding an elliptic element $\gamma$, we introduce the relation $\gamma=1$ in the local isotropy
group at $x$. If $x$ corresponds to a $S^2(2,2,2,2)$ cusp, then the new isotropy group at $x$ is a quotient of the Klein 4 group. If not, then
$x$ corresponds to a trivalent vertex of the isotropy graph, say each edge corresponds to torsion elements $\gamma$, $a$, and $b$. Introducing the relation that $\gamma=1$, 
to $ab\gamma=1$ 
 yields $a=b^{-1}$. In particular $a$ and $b$ have the same order. Therefore, in the image $f(\Gamma_Q)$, $f(a)=f(b^{-1})$ and graphically
we can relabel the weights corresponding to $a$ and $b$ with $gcd(|a|,|b|)$ (see 
Fig \ref{fig:GraphCuspKill}). 
Relabeling the edges could introduce further reductions to the graph, however
since the isotropy graph has a finite number of vertices and edges and each edge is weighted by a finite integer, this process will terminate in a finite number of steps.
\end{rem}

\section{Integral representations of knot groups}\label{sect:knotReps}

In this section, we provide two lemmas about the representations of knot groups into $\PSLTA$.
Both main theorems rely on these two lemmas in order to limit the possible covers for $p\co\Sth-K\rightarrow Q$, where $\GK$ admits integral traces (see $\S$\ref{sect:background}) and $Q$ has a rigid cusp. 

The key tool from number theory used in the proofs of these lemmas is the Hilbert class field. For further background on the Hilbert class field (see \cite[$\S$10.2]{Koch2000}). This field is special in two ways. First, if $E$ is the Hilbert class field of $L$, then any ideal $I$ in $\mathcal{O}_L$ is principal in $\mathcal{O}_E$. The second fact that we exploit is that $E$ is an unramified extension of $L$. By this we mean that for each prime ideal $P$ of $\mathcal{O}_L$, $P$ either remains prime or $P$ splits into product of distinct prime ideals in $\mathcal{O}_E$, that is $P=\Pi_{i=1} ^m Q_i$ where each $Q_i$ is prime in $\mathcal{O}_E$ and $Q_i \ne Q_j$ for all $i \ne j$.  

\begin{lem}\label{lem:intTracesUpperTri}
Let $\GK$ be a knot group admitting integral traces. Then for any $Q$ such that $\Sth-K$ covers $Q$, $\GQ$  admits a representation into $\PSLTA$ where $\Gamma_K$ has an upper triangular meridian.
\end{lem}

\begin{proof}
Since $\GK$ admits integral traces, $\GQ$ admits integral traces. Therefore, there is an integral representation of $\GQ$ into $PSL(2, \mathcal{O}_L)$ where $L$ is the Hilbert class field of the trace 
field of $\GQ$ (\cite[Lem 5.2.4]{MR}).  

Considering a Wirtinger presentation for $\GK\subset \GQ$, we may assume $\GK$ is generated by parabolics. Hence, 

$$\GK = \langle \mu_1, \mu_2, ... \mu_n \rangle \mbox{ where }\mu_i = \begin{pmatrix} 1+\beta_i\gamma_i & -\beta_i^2 \\ \gamma_i^2 & 1-\beta_i\gamma_i \end{pmatrix}.$$

We claim that we can conjugate $\GQ$ such that $\GQ$ remains in $\PSLTA$ with $\mu_1$ upper triangular. For ease of notation, we will suppress the subscript notation in the following
argument. Note that here $\mu$ fixes $\frac{\beta}{\gamma}$.  If $\beta \mathcal{O}_L$ and  $\gamma \mathcal{O}_L$
are comaximal, ie there exist $p,q \in \mathcal{O}_L$ such that $p\beta+q\gamma=1$, then 
conjugating $\GQ$ by $h =  \begin{pmatrix} p & q \\ -\gamma & \beta \end{pmatrix}$ yields an upper triangular meridian $h \mu h^{-1}$. 
Since $h$ has determinant 1, $h\GQ h^{-1} \subset \PSLTA$.
If $\beta \mathcal{O}_L$ and  $\gamma \mathcal{O}_L$ are not comaximal, then $I=\langle \beta, \gamma \rangle$ is a proper ideal. If $I$ is principal, then we may say that
$\beta = r \alpha$, $\gamma= s \alpha$ for some $\alpha\in\mathcal{O}_L$. Since $\mathcal{O}_L$ is a Dedekind domain, there exist $p,q \in \mathcal{O}_L$, $p\beta+q\gamma=\alpha$ and $pr+qs=1$. Therefore, conjugating $\GK$ by $h' =  \begin{pmatrix} p & q \\ -s & r \end{pmatrix}$
yields the desired representation. If $I$ is not a principal ideal in $\mathcal{O}_L$, then we may pass to the Hilbert class field of $L$, say $E$ and apply the same argument since
$I$ is a principal ideal in $\mathcal{O}_E$.
\end{proof}

The following lemma helps specify what values can go in the upper right entry of the meridian exhibited above. 

\begin{lem}\label{lem:xIsAUnit}
Let $\GK$ be a knot group in $\PSLTA$ with an upper triangular meridian $\mu=\begin{pmatrix} 1 & x \\ 0 & 1 \end{pmatrix}$. 
Then $x$ is a unit in $\mathbb{A}$. Furthermore, we may assume that $\mu=\begin{pmatrix} 1 & 1 \\ 0 & 1 \end{pmatrix}$. 
\end{lem}

\begin{proof}
Since $\GK$ is finitely generated, we can attach all of the entries of the generators of $\GK$ to $\Q$ and get a finite 
extension of $\Q$. Denote this extension by $L$.  Hence, $\GK \subset PSL(2,\mathcal{O}_L)$.

Assume $x$ is not a unit. Then $x\mathcal{O}_L$ is a proper ideal and we may say $x\mathcal{O}_L = \Pi_{i=1} ^n P_i ^{e_i}$ where each $P_i$ is a maximal ideal in $\mathcal{O}_L$. 
Let $J=\Pi_{i=1} ^n P_i$ for the same $P_i$.

Denote by $E$ the Hilbert class field of $L$. Then $J = \alpha\mathcal{O}_E$. Since $E$ is an unramified extension of L and each of the $P_i$ are maximal ideals in $\mathcal{O}_L$, $J = \Pi_{i=1} ^m Q_i$
where $Q_i \ne Q_j$ if $i \ne j$. 

Let $g=\begin{pmatrix} a & b \\  c & d   \end{pmatrix} \in \GK$ and denote by $h=\begin{pmatrix} \sqrt{\alpha} & 0 \\ 0 & \frac{1}{\sqrt{\alpha}} \end{pmatrix}$. 
Then, $$h\cdot g \cdot h^{-1}= \begin{pmatrix} a & b\cdot\alpha \\  \frac{c}{\alpha}  & d   \end{pmatrix}.$$

We claim $\frac{c}{\alpha}$ is in $\mathcal{O}_E$. To see this it suffices to show $c$ is in $J$. There exists a homomorphism $f_i \co PSL(2,\mathcal{O}_E) \rightarrow PSL(2,\mathcal{O}_E/Q_i)$ 
for each prime ideal $Q_i$. Under such a homomorphism $f_i(\GK)$ is trivial, since $f(\mu)$ is trivial and $\mu$ normally generates $\GK$. Therefore, 
$f_i(g)= \begin{pmatrix} 1 & 0 \\  0 & 1   \end{pmatrix}$  and
$c$ is in $Q_i$ for each $Q_i$. Hence, $c$ is in $J$.

Also, $h \cdot \mu \cdot h^{-1} = \begin{pmatrix} 1 & x\alpha \\ 0 & 1 \end{pmatrix}$. Note, that $x\alpha \mathcal{O}_E$ and $x \mathcal{O}_E$ factor into the same set of prime ideals,
if we ignore the multiplicities of the factors. Hence, the construction of $J$ is independent of starting with $\GK$ or $h\GK h^{-1}$.  
Therefore, by applying the above argument to $h \GK h^{-1}$, $\frac{c}{\alpha} \in J$ as well. 

Let $$V(g) = min\left\{ v_{Q_i}(c) | g=\begin{pmatrix} a & b \\ c & d \end{pmatrix} 
\right\}.$$ Here $v_{Q_i} (c)$ denotes the power of $Q_i$ in the factorization of the ideal generated by $c$.  
Notice that $V(h\cdot g \cdot h^{-1})=V(g)-1$. 

There is an element $g' \in \GK$ such that $V(g')=s$ is minimal. Therefore $V(h^s \cdot g' \cdot h^{-s})=0$, but $h^s \cdot \GK \cdot h^{-s}$ is trivial since 
$h^s \cdot \mu \cdot h^{-s} = \begin{pmatrix} 1 & x\alpha^s \\ 0 & 1 \end{pmatrix}$ normally generates $h^s \cdot \GK \cdot h^{-s}$. This is a contradiction.

Finally, since $x$ is a unit, we may conjugate $\GK$ by $h'=\begin{pmatrix} \frac{1}{\sqrt{x}} & 0 \\ 0 & \sqrt{x} \end{pmatrix}$ so that there is a meridian of the 
form $ \begin{pmatrix} 1 & 1 \\ 0 & 1 \end{pmatrix}$, while preserving the integrality of the representation.
\end{proof}

\begin{rem}\label{rem:betaGammaComaximal}
As a consequence of the argument above with $\GK$ in $PSL(2,\mathcal{O}_L)$, the meridian $\mu= \begin{pmatrix} 1+\beta\gamma & -\beta^2 \\ \gamma^2 & 1-\beta\gamma \end{pmatrix}$ must have the property 
that $\beta \mathcal{O}_L$ and  $\gamma \mathcal{O}_L$
are comaximal. Otherwise, after conjugation the upper triangulation meridian will not have a unit in the upper right entry.
\end{rem}

\section{Proof of Theorem \ref{mainthm1}}\label{sect:Thm1Proof}

In this section, we prove Theorem \ref{mainthm1}. In the proof, we assume that 
$\Sth-K$ cyclically covers an orbifold with a torus cusp, which we denote by $Q_T$ throughout the section.  The main argument is that $Q_T$ cannot admit multiple cyclic fillings because $Q_T$ must be a relatively high degree cover of a rigid cusped orbifold. 

We say the isotropy graph of a orbifold has a \emph{loop} if there is an edge in the isotropy graph that begins and ends at the same vertex. It is a consequence of the cusp killing homomorphism that loops must begin and end the cusp of orbifold covered by a knot complement. 

Also, we note that Propositions \ref{prop:abelianQuotients} and \ref{prop:nofourloop} are also used in the proof of Theorem \ref{mainthm2}. 

We begin by classifying the possible abelian quotients of rigid cusped orbifolds covered by hyperbolic knot complements.

\begin{prop}\label{prop:abelianQuotients}
Let $\Sth-K$ be a hyperbolic knot complement that covers an orbifold Q.
\begin{enumerate}
\item If Q has a $S^2(2,3,6)$ cusp, then $\Z/2\Z$ surjects $\GQ ^{ab}$.
\item If Q has a $S^2(3,3,3)$ cusp, then $\Z/3\Z \times \Z/3\Z$ surjects $\GQ ^{ab}$.
\item If Q has a $S^2(2,4,4)$ cusp, $\GQ ^{ab}$ is trivial, $\Z/2\Z$, $\Z/2\Z\times\Z/2\Z$, or $\Z/4\Z$.
Furthermore, $\GQ ^{ab} \cong \Z/4Z$ if and only if the isotropy graph of $Q$ has a loop.
\end{enumerate}
\end{prop}

\begin{proof}
We first note that $\GQ = P_Q \cdot \GK$. Therefore, we claim that
$$\GQ=\left\langle\left\langle
t=\begin{pmatrix} 1 & 1\\ 0 & 1 \end{pmatrix},
r=\begin{pmatrix} \ell & 0\\ 0 & \ell^{-1} \end{pmatrix}
\right\rangle\right\rangle_{P_Q\cdot\GK},$$ 
where $r$ is an elliptic element of order 3, 4, or 6 depending on the cusp type.
The claim follows from the fact that $r$ and $t$
are generators for $P_Q$ and
 $P_Q$ contains a meridian of $\GK$. 

First assume Q has a $S^2(2,3,6)$ cusp, then $\ell = e^{\frac{i\pi}{6}}$.
Since $P_Q$ abelianizes to $\Z/6\Z$, 
we know $\GQ^{ab}$ is a quotient of
$\Z/6\Z$. Also, the torsion element of order 6 is connected to an interior
vertex with the isotropy group a $D_6$, (dihedral group of order 12). Under the abelianization
of this isotropy group, the element of order 6 maps to an element of order 2.
Thus, $\GQ^{ab}$ is a quotient of $\Z/2\Z$, as desired.

Next assume Q has a $S^2(3,3,3)$ cusp. In this case, $\ell = e^{\frac{2i\pi}{3}}$.
Then $\GQ^{ab}$ is a quotient of $\Z/3\Z \times \Z/3\Z$, the abelianization of the peripheral subgroup (see \ref{sect:rigidCusps}).

Finally assume Q has a $S^2(2,4,4)$ cusp, then $\ell=e^{\frac{i\pi}{4}}.$
Hence, $\GQ^{ab}$ is a quotient of the abelianization of $P_Q$, which is $\Z/2\Z \times \Z/4\Z$. 

Consider an edge $e$ labeled by 4-torsion that connects the cusp $c$ to another vertex $x$. Then $x$ is either
the cusp itself, or it corresponds to an isotropy group, $D_4$ or $S_4$ (see $\S$\ref{sect:SubGroupsofSOn}).

\textbf{Case 1:} $e$ is a loop.
In this case, $e$ connects the cusp back to itself. Then, considering a Wirtinger presentation for $\GQ$ coming from
the isotropy graph. $\GQ^{ab}$ is either $\Z/4\Z$ or $\Z/4\Z\times\Z/2\Z$. However, the 
second case cannot occur since this would imply the existence of a path labeled by only even
numbers that starts and ends at the cusp 
 (by the cusp killing homomorphism) and that also includes the 
2 torsion on the cusp. Such a cycle would kill the 4 torsion on the cusp and the maximal abelian quotient would be 
$\Z/2\Z \times \Z/2\Z$. 

\textbf{Case 2:} $x$ corresponds to $D_4$ or $S_4$. 
Under the abelianizations of these groups, elements of order 4 are mapped to elements of order 2. 
Under the abelianization of the cusp, the peripheral elements of order 4 all have the same order in $\GQ^{ab}$.
Therefore, $\Z/2\Z \times \Z/2\Z$ surjects $\GQ^{ab}$. In this case, the isotropy graph of $Q$ does not 
contain a loop. This completes the proof.
\end{proof}

In the following propositions lemmas, we appeal to notation and definition from \cite{BBCW}. In this paper, they define an \emph{orbilens space} to be the quotient of $\Sth/\Delta$ where $\Delta$ is a
finite cyclic subgroup of $SO(4,\mathbb{R})$. As noted in that paper, for any orbilens space $L$,
$|L|$ is a lens space $L(p,q)$. Furthermore, there exists a Heegaard splitting of $|L|$ such that $L$ decomposes into two pieces each of which is the quotient of solid torus under rotation about its core. Using $p,q$ to denote the underlying lens space and $n$ and $m$ to denote the orders of these rotations, we use the notation $L(p,q;n,m)$ to denote an orbilens space. Finally, if two knot complements are cyclically commensurable, they both cover the complement knot in an orbilens space (see \cite[Prop 4.13]{BBCW}). 

The lower bound on the degree of a manifold cover exhibited by this proposition will be used in the next section. However, the classification of knot complements in orbilens spaces that 4-fold cover rigid cusped orbifolds with loops will be useful in this section.

\begin{prop}\label{prop:nofourloop}
If $p\co M\rightarrow Q$ with a $S^2(2,4,4)$ cusp,
the isotropy graph for Q has a loop labeled 4, and $M$ is covered by a knot complement, then
$deg(p)\geq 24$. Furthermore, if there is a loop labeled 4, then there exists $f\co Q_T \rightarrow Q$ where
$Q_T$ is a knot complement
in an orbilens space $O$ and $deg(f)=4$, where the singular locus of $O$ is two unknotted circles. 
\end{prop}

\begin{figure}
\centering
\resizebox{2.5 in}{!}{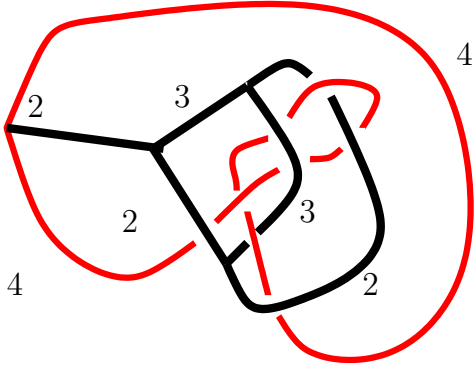}
\caption{\label{fig:GraphWith4Loop} A possible isotropy graph of an orbifold with a $S^2(2,4,4)$ cusp}
\end{figure}

\begin{proof}

Assume that $Q$ is covered by a knot complement and has an edge labeled 4 in its 
isotropy graph (see Fig \ref{fig:GraphWith4Loop} for a possible example). 
Then $\GQ^{ab}$ is $\Z/4\Z$ (see Prop \ref{prop:abelianQuotients}). 
Hence, there is a unique orbifold $Q'$ which has a $S^2(2,2,2,2)$ cusp and is 2-fold cover of $Q$.

Since $P_{Q'}$ and $P_Q$ have the same parabolic subgroup, $P_{Q'}$ contains a meridian $\mu$ of $\GK$.
The abelianization of $\GQ$ is $\Z/4\Z$ so $\Gamma_{Q'}$ is characteristic in $\GQ$. In particular,
$\langle \langle \mu \rangle \rangle_{\GK} \subset \Gamma_{Q'}$.  Therefore,
$\Sth-K$ covers $Q'$.
Also, $Q'$ has $S^2(2,2,2,2)$ cusp and the cover $p':\Sth-K\rightarrow Q'$ is a regular covering (see \cite[Prop 9.2]{NR1}).
Furthermore, there is a unique 2-fold cover of $Q'$ that has a torus cusp. We denote this orbifold by $Q_T$.
We know that $\Sth-K$ covers $Q_T$ by an identical argument to that above.

Since $Q_T\cong (\Sth-K)/Z$ for some cyclic group $Z$, we see that $Q_T$ is the complement of
 a knot in an orbi-lens space. The isotropy graph for $Q_T$ is either 0, 1, or 2 unknotted
 circles (see \cite[Lem 3.1]{BBCW}). Therefore, the isotropy graph for $Q'$ contains 0, 2, or 4 internal vertices.

 \begin{figure}
\centering
\subfigure[2 vertex graph]{
\resizebox{2 in}{!}{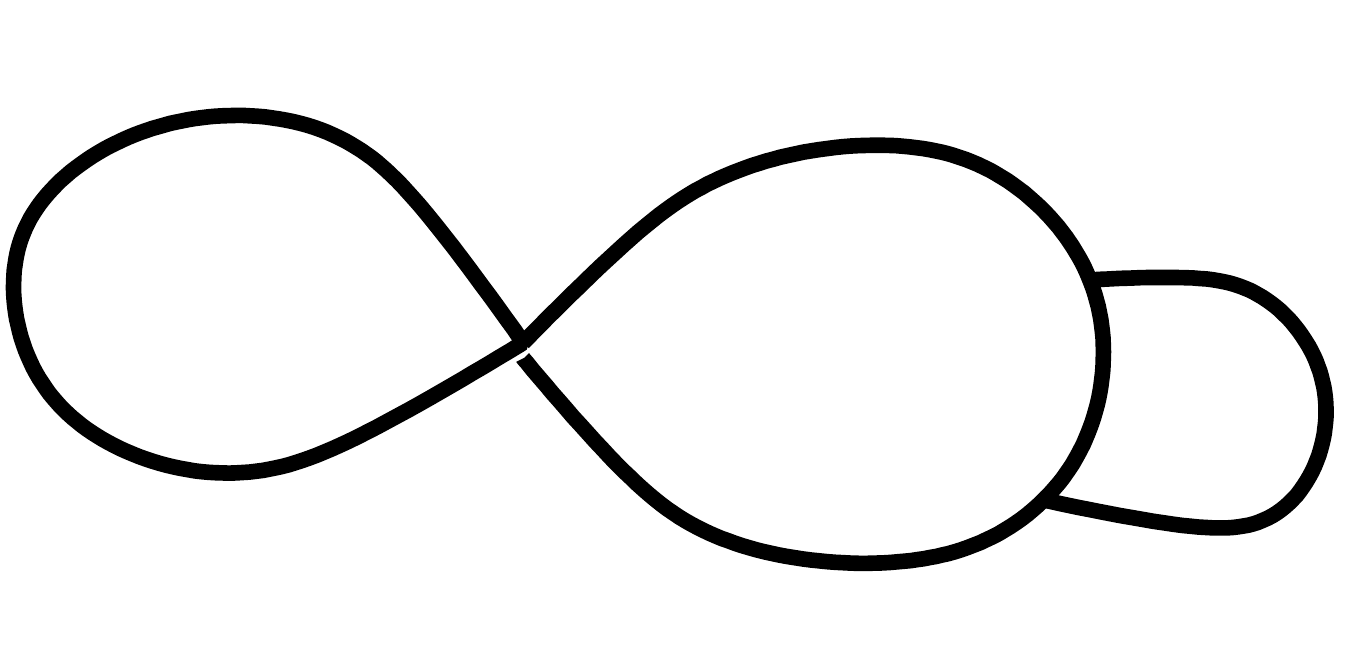}
}
\subfigure[4 vertex graph]{
\resizebox{2 in}{!}{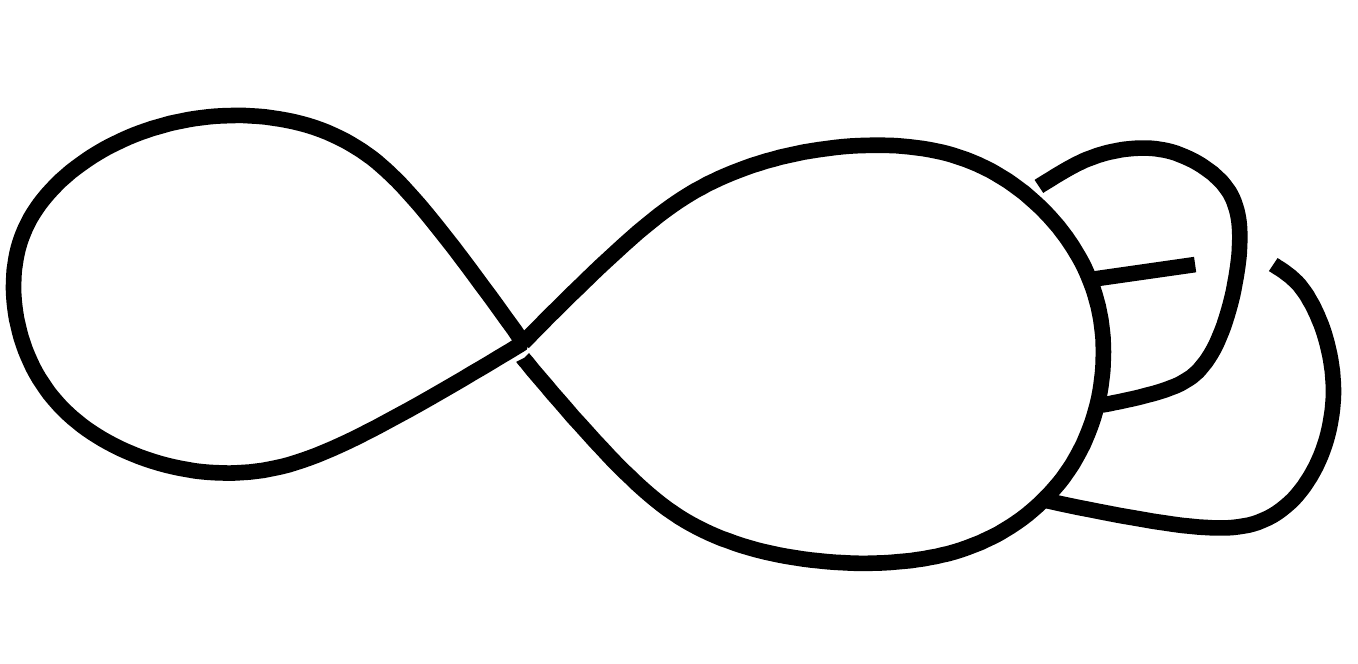}
}
\caption{\label{fig:IsotropyGraphsOfQ} The possible isotropy graphs $Q'$}
\end{figure}

The isotropy graph for $Q'$ cannot contain 0 vertices because that would imply that the isotropy graph for $Q$
only had vertices labeled by Klein 4 groups. Such a graph would be non-trivial under the cusp killing homomorphism.
Since the 2-fold of the loop labeled by 4 is a loop labeled by 2, the possible graphs as defined up to graph isomorphism type can be see Figure \ref{fig:IsotropyGraphsOfQ}.

We claim that the isotropy graph for $Q'$ cannot contain 2 internal vertices as well. 
First, notice that the isotropy graph for $Q'$ has an edge $e$ labeled by 2-torsion with both endpoints on the cusp.
Thus, the isotropy graph $Q$ takes the form of Figures \ref{fig:FourLoopCase1},  \ref{fig:FourLoopCase2},  and 
\ref{fig:FourLoopCase3}.  In the latter two cases, the orbifold is non-trivial under the cusp killing homomorphism and therefore cannot be covered by a knot complement. In the first case, we cannot
close up the isotropy graph.  Therefore, no such orbifold $Q$ can be covered by a knot complement. 

Finally, if $Q'$ contains 4 internal vertices, then $Q_T$ is the complement of a knot in an orbi-lens space with an isotropy graph consisting of two unknotted circles.
Since the circles are labeled by $n$-torsion and $m$-torsion with $n,m\geq 2$ and $(n,m)=1$, $\Sth-K$ is at least a 6-fold cover of $Q_T$ and therefore at least a
24-fold cover of $Q$. This completes the proof. 
\end{proof}

\begin{figure}[h]
	\centering
\subfigure[The first case
]{	\resizebox{1 in}{!}{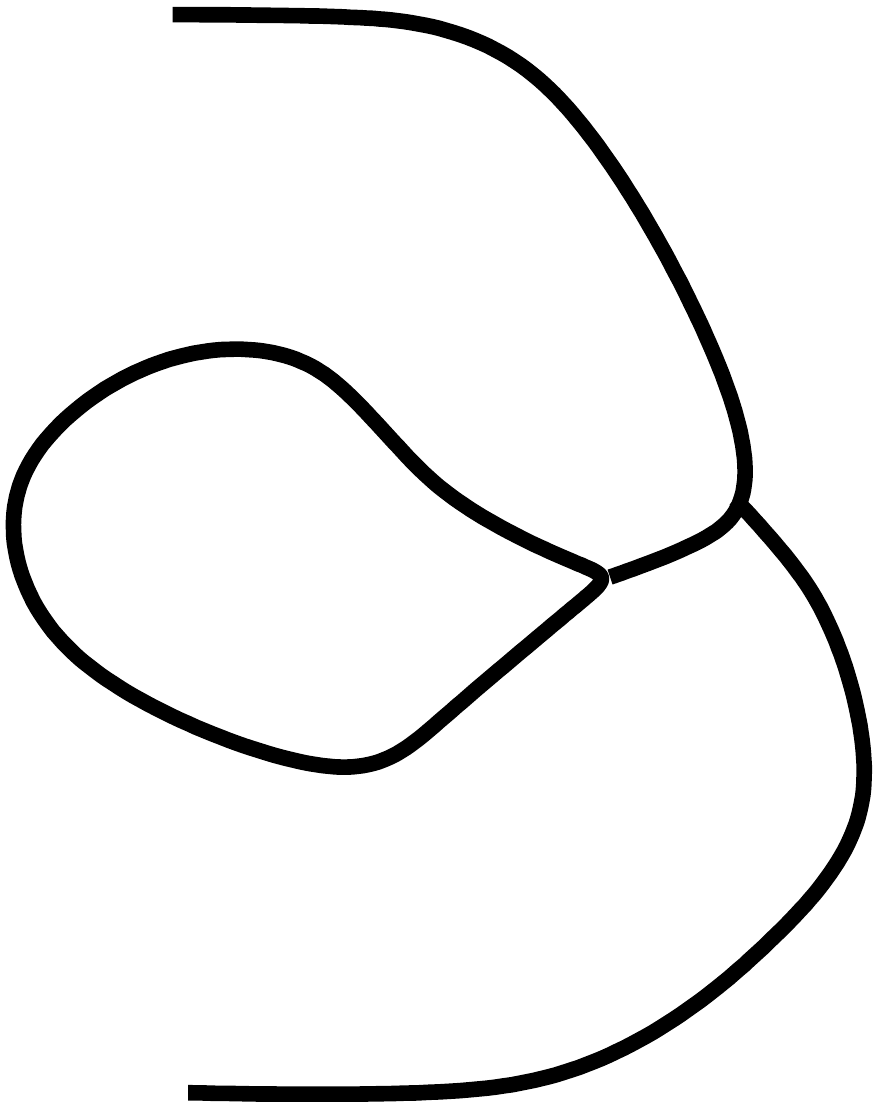}
	\label{fig:FourLoopCase1}
}
\subfigure[The second case]{
	\resizebox{1.25 in}{!}{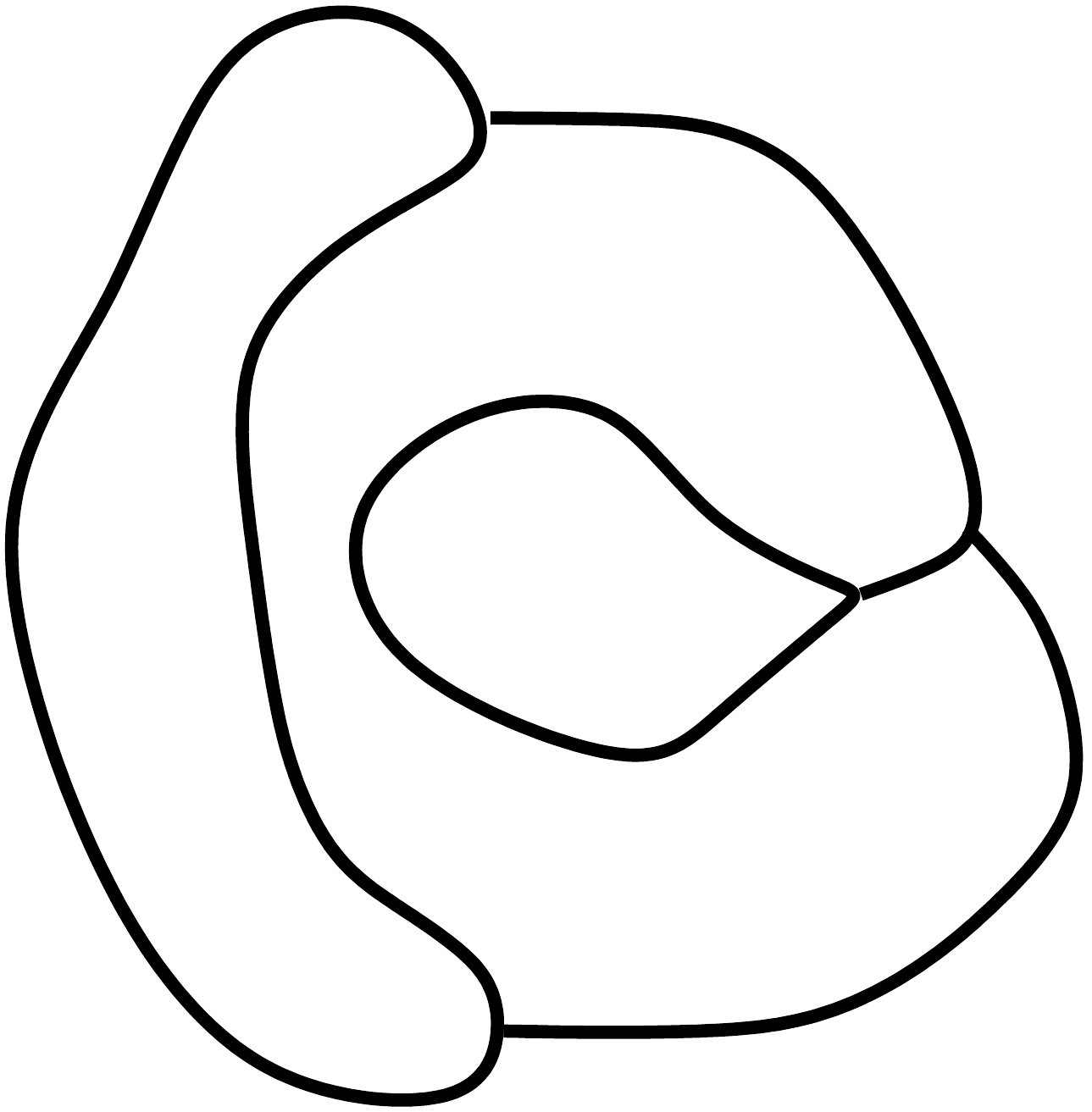}
	\label{fig:FourLoopCase2}
}
\subfigure[The third case]{
	\resizebox{1.25 in}{!}{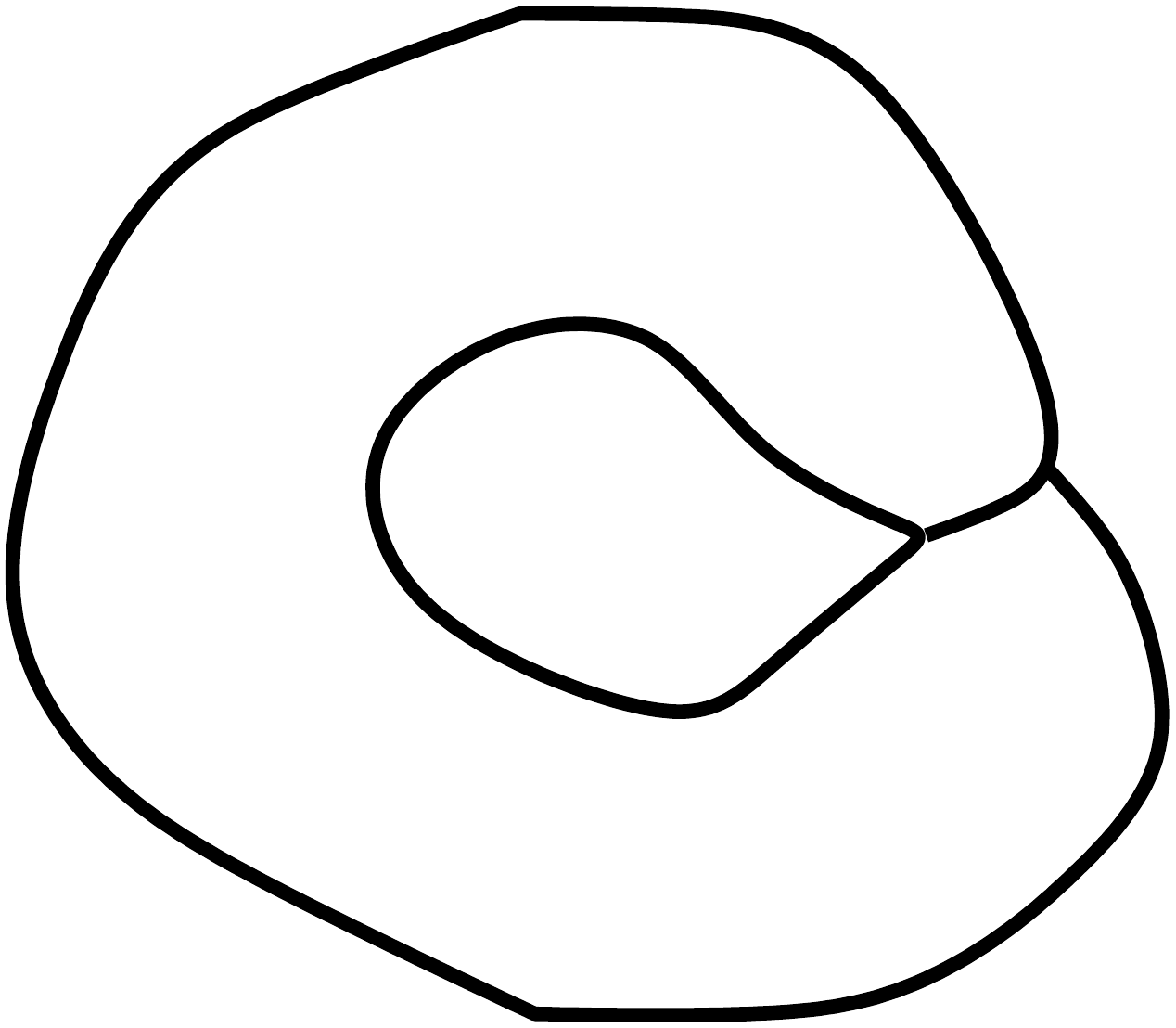}
	\label{fig:FourLoopCase3}
}
\caption{\label{fig:ThreeCases4loop} The three cases for an isotropy graph with 2 vertices}
\end{figure}

\begin{lem}\label{lem:lowerOrbLensBound}
Let $\Sth-K$ is small knot complement that covers an orbifold $Q_T$ with a torus cusp. Furthermore assume $Q_T$  covers  $Q_r$ with a rigid cusp with covering degree $d$.
\begin{enumerate}
\item If $Q_r$ has a $S^2(3,3,3)$ cusp, then $d > 3$.
\item If $Q_r$ has a $S^2(2,3,6)$ cusp, then $d > 6$.
\item If $Q_r$ has a $S^2(2,4,4)$ cusp, then $d>4$ or $Q_T \cong L(p,q:n,m)-K'$ for some embedded knot $K'$ and some $n\ne 1$ and $m\ne 1$.

\end{enumerate}
\end{lem}

\begin{proof}
\textbf{Case 1: $Q_r$ has a $S^2(3,3,3)$ cusp.}
First note that $d\geq 3$ since $Q_r$ has 3-torsion on the cusp and $Q_T$ has a torus cusp.
If $d=3$, then all isotropy groups must lift cyclic groups. From inspection of the possible isotropy groups, all groups must be $D_3$ since no other candidates have index 3 cyclic groups. In this case, $\Gamma_r$ would be non-trivial under the cusp killing map, so it cannot be covered by a knot complement. Thus, $d>3$.

\textbf{Case 2: $Q_r$ has a $S^2(2,3,6)$ cusp.}
By a similar argument to that above $d \geq 6$. If $d=6$, then $[\Gamma_r:\Gamma_T]=6$ and there is a homomorphism from $\Gamma_r$ to $S_6$ via left multiplication on the cosets of $\Gamma_T$. Denote by $r$ an element of 6-torsion that fixes a point on the cusp. Then, 
$\Gamma_r = {\Gamma_T, r\cdot \Gamma_T, r^2\cdot \Gamma_T,r^3\cdot \Gamma_T,r^4\cdot \Gamma_T,r^5\cdot \Gamma_T}$.  Hence, $r$ maps to an element of order 6 and the image of $\Gamma_r$ is not a subgroup of  $A_6$. Therefore, $\Gamma_r$ admits a $\Z/2\Z$ quotient and a unique subgroup $\Gamma_{r'}$ of index 2 which contains $\Gamma_T$. Such a group corresponds to an orbifold with a $S^2(3,3,3)$ cusp and so we have reduced the problem to the previous case. 

\textbf{Case 3: $Q_r$ has a $S^2(2,4,4)$ cusp.}
If $d=4$, then the 4-torsion either forms a loop, is connected to a $D_4$ or $S_4$ isotropy group. The first case is covered by Proposition \ref{prop:nofourloop}. In the second case, some 2-torsion must survive the lift. However, the 2-torsion on the cusp must be a part of a $D_{2n+1}$ in order to lift to a cyclic group and allow for $\Gamma_r$ to be non-trivial under the cusp killing map. In the final case, $S_4$ has no cyclic subgroups of index 4. 

\end{proof}

We are now ready to prove Theorem \ref{mainthm1}.

\begin{proof}[Proof of Theorem \ref{mainthm1}.]
Since $\Sth-K$ is small, $\GK$ admits integral traces (see \cite{Bass1980}). Also, all groups commensurable with $\GK$ admit integral traces.   
In particular, $\GQ$ admits integral traces and therefore a representation into $\PSLTA$ (see \cite[Lem 5.2.4]{MR}). 

Let $\Sth-K'$ be a knot complement that is cyclically commensurable with $\Sth-K$. Therefore, 
both knot complements cover an orbifold $Q_T$, which is a knot complement in an orbilens 
space and the maximal volume orbifold covered by both  $\Sth-K'$  and  $\Sth-K$   (see \cite
[Prop 4.7]{BBCW}). We may assume that $Q$ is the orientable commensurator quotient and 
therefore both $\Sth-K$ and $\Sth-K'$ cover $Q$. Hence, $Q_T$ covers $Q$. Using the 
representation for $\GQ$, we know that $\Gamma_T$ admits an integral representation and 
therefore we may assume that $P_T$ is upper triangular and contains two distinct knot 
meridians in it (see Lem \ref{lem:intTracesUpperTri}). Therefore there are two units in upper 
right entries of these meridians (see Lem \ref{lem:xIsAUnit}). Hence, $P_T$ is the minimal 
index abelian subgroup of $P_Q$. Since the index of $P_T$ in $P_Q$ is the degree of the 
cover $p:Q_T \rightarrow Q$, this contracts Lemma \ref{lem:lowerOrbLensBound} unless $Q$ 
has a $S^2(2,4,4)$ cusp and $Q_T\cong L(p,q;n,m)-K'$. In this case, $n\ne 1$ and $m\ne 1$. Hence, $Q_T$ cannot be covered by two knot complements (see \cite[Thm 1.8]{BBCW}). 
This completes the proof.
\end{proof}

\section{Proof of Theorem \ref{mainthm2}}\label{sect:Thm2Proof}

In this section, we prove Theorem \ref{mainthm2}.  The first step of the proof involves two 
arguments 
that orbifolds with 
certain prescribed cusp volumes cannot be covered by knot complements. The second 
ingredient 
is lower bound on the degree of the covering $p\co M \rightarrow Q$ where $M$ is a 
manifold covered by a small knot complement and $Q$ is a rigid cusped orbifold.  
The third part of the proof establishes an upper bound on $deg(p)$ by combining the Six Theorem of Agol and Lackenby (see \cite{Agol6Theorem}, \cite{Lackenby6Theorem}) with the lemmas of $\S$\ref{sect:knotReps}. 

We use the cusp killing homomorphism (see Prop \ref{cusp_kill}) to prove the following lemma. Also, in the following proofs, we will identify $\Hth$ with $\{z+cj \in \mathbb{H} | z \in \mathbb{C}, c>0, j^2=-1\}$ 
(upper-half space)
and $\partial \Hth$ with $\mathbb{C} \cup \{\infty\}$. Here, we denote by $B_x$ the horoball that is tangent to $\partial \Hth$ at $x$.

\begin{lem}\label{nosqrt21over12}
\begin{enumerate}
\item Any orbifold  with a $S^2(3,3,3)$ cusp and cusp volume $\frac{\sqrt{21}}{12}$ cannot be covered by a knot complement.
\item Any orbifold  with a $S^2(2,3,6)$ cusp and cusp volume $\frac{\sqrt{21}}{24}$ cannot be covered by a knot complement.
\end{enumerate}
\end{lem}

\begin{proof}
First, we appeal to Adams' characterization of orbifolds of small cusp volume (see $\S$\ref{sect:rigidCusps}) to 
reduce to case (2) as any orbifold with a $S^2(3,3,3)$ cusp and  cusp volume $\frac{\sqrt{21}}{12}$ covers an orbifold 
with $S^2(2,3,6)$ cusp and cusp volume $\frac{\sqrt{21}}{24}$. 

Hence, let $Q$ be a 3-orbifold with a $S^2(2,3,6)$ cusp and cusp volume $\frac{\sqrt{21}}{24}$.
A diagram of the horoballs associated to $Q$ first appeared in Adams' paper (see \cite[Fig 5]{Adams1}), 
however we include it here as Figure \ref{fig:Sqrt21over12horoballs}
for the sake of completeness. Furthermore, following the discussion
of this horoball diagram in \cite{Adams1}, we use the following notation: 
$O=0$, $D=\sqrt[4]{7}$, $X=\frac{5+i\sqrt{3}}{2\sqrt{7}}$,  and $Y=\frac{\sqrt[4]{7}}{2}+i\frac{\sqrt[4]{7}}{2\sqrt{3}}$. 

In this figure, there are four horoballs pictured. Following the description of this diagram from Adams' work, 
the horoballs $B_O$ and $B_D$ are of Euclidean diameter 1 and maximal in the 
sense that they are tangent to the horoball based at $\infty$. The horoball $B_X$ has Euclidean diameter
$\frac{1}{\sqrt{7}}$ and the horoball $B_Y$ has Euclidean diameter $\frac{3}{7}$. 
The line segment $\overline{OY}$ is length $w=\frac{\sqrt[4]{7}}{\sqrt{3}}$ while the line segment $\overline{OX}$ is length
$\frac{1}{\sqrt[4]{7}}$.  In particular, $\frac{\sqrt[4]{7}}{\sqrt{3}}\approx 0.939104416 < 1$.

\begin{figure}
\centering
\resizebox{8cm}{!}{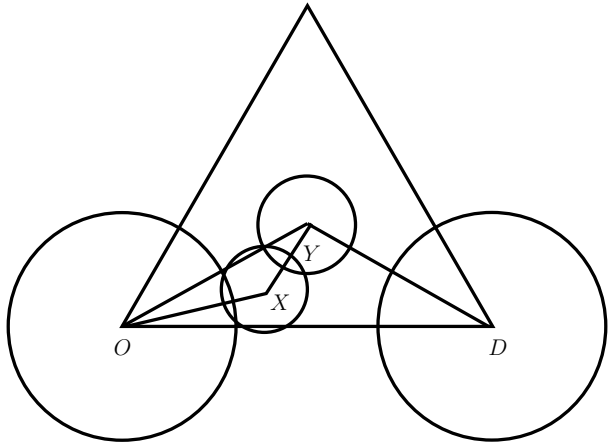}
\caption{\label{fig:Sqrt21over12horoballs} A horoball diagram for $Q$}
\end{figure}

Under Adams' description of $Q$, we see that $\Gamma_Q$ contains a parabolic element $t$ 
such that $t(\infty)=\infty$ and $t(0)=\sqrt[4]{7}$. In addition, $\Gamma_Q$ contains 
an order 6 rotation $r$ that fixes $0$ and $\infty$. 
Finally, as Adams notes, all horoballs of Euclidean diameter 1 are equivalent under the action of $P_Q$. 
Therefore, there is an element $\gamma$ 
that exchanges $0$ and $\infty$ while sending $\sqrt[4]{7}$ to $\frac{5+\sqrt{-3}}{2\cdot\sqrt[4]{7^3}}$ 
(see \cite[Lem 1.2]{Adams1}). 
Hence,
$$t = \begin{pmatrix} 
 1 & \sqrt[4]{7}\\
0 & 1
\end{pmatrix},
\mbox{ }
r =\begin{pmatrix} 
\ell & 0\\
0 & \ell^{-1}
\end{pmatrix},
\mbox{ and }
\gamma = \begin{pmatrix} 
 0 & i\cdot b\\
\frac{i}{b} & 0
\end{pmatrix}
$$
where $\ell = \frac{\sqrt{3}+i}{2}$ and $b =\frac{\sqrt{5+i\sqrt{3}}}{\sqrt{2\sqrt{7}}}$.

The isometric sphere of $\gamma$ is of radius 1 and centered at 0 (see Fig \ref{fig:Sqrt21over12IsoSpheres}). Hence, the isometric sphere for $t\cdot\gamma\cdot t^{-1}$ is radius 1 and centered at $\sqrt[4]{7}$. Let $\Gamma = \left\langle t, r, \gamma \right\rangle$. Since these two isometric spheres bound a fundamental domain for $\Gamma$ away from $\mathbb{C}$, 
$\Gamma$ has finite co-volume. Also, since the cusp co-volume of $\Gamma$ is $\frac{\sqrt{21}}{24}$, $[\Gamma_Q:\Gamma]=1$.

\begin{figure}
\centering
\resizebox{9cm}{!}{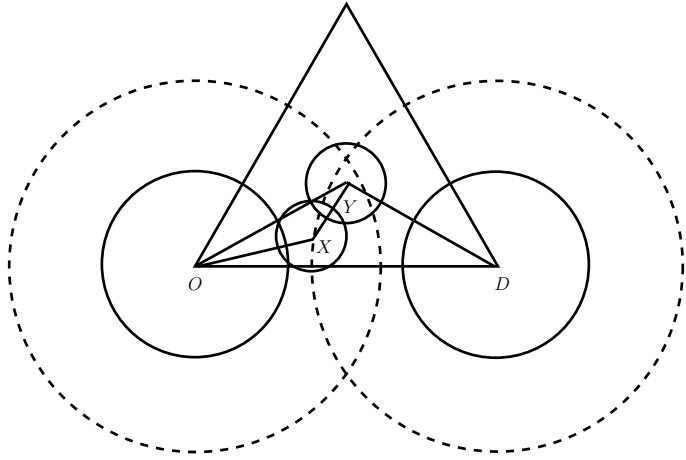}
\caption{\label{fig:Sqrt21over12IsoSpheres} A horoball diagram for $Q$ with the isometric spheres represented by dotted curves around $0$ and $D$}
\end{figure}

Let $\lambda = \frac{1}{\sqrt{b}}$ and 
$$g = \begin{pmatrix}
 \lambda & 0\\
0 & \lambda^{-1}
\end{pmatrix}.
$$

Then, 
$$g\cdot \Gamma \cdot g^{-1} = \left\langle  t'=\begin{pmatrix}
 1 & \alpha\\
0 & 1
\end{pmatrix},   \begin{pmatrix}
\ell & 0\\
0 & \ell^{-1}
\end{pmatrix},
\gamma' = \begin{pmatrix}
0 & i\\
i & 0
\end{pmatrix}
\right\rangle$$
where $\alpha = \sqrt{\frac{14}{5+i\sqrt{3}}}$.  Note the minimal polynomial for $\alpha$ is $q^4-5q^2+7=0$. Thus, $\alpha$ is an algebraic integer, but not a unit
since the constant term of this polynomial is not 1.
  Under this integral representation of $\Gamma_Q$, there are upper parabolic elements. 
  Hence, a knot complement covering $Q$ would contradict Lemma \ref{lem:xIsAUnit}.
\end{proof}

The following proposition discusses which finite groups can act on a point of tangency between two horoballs.

\begin{prop}\label{propD3orC3}
Consider a maximal horoball packing corresponding to an orbifold with a rigid cusp. 
Denote by $B_x$ the horoball centered x and denote by $B_{\infty}$ the horoball at $\infty$.
If $y$ is the point of tangency of $B_x$ and $B_{\infty}$ and $y$ fixed by an 
element $\gamma \in Stab_{\infty}$, then
the isotropy group of y is $C_n$ or $D_n$ where $n=2,3,4,6$. 
\end{prop}

\begin{proof}
First $\gamma$ is order 2,3,4, or 6 because it fixes $\infty$.
If $Stab_y$ is cyclic or dihedral, we are done.  

Denote by $\gamma'$ be an element of the isotropy group of $y$ such that axis fixed by $\gamma'$ intersects the
axis fixed  $\gamma$ at the smallest angle possible. Denote this angle by $\alpha$. If $\alpha=\frac{\pi}{2}$, then, 
$\langle \gamma,\gamma' \rangle$ is dihedral (see $\S$\ref{sect:SubGroupsofSOn}). Hence, we may assume that $\alpha<\frac{\pi}{2}$.
Therefore, $\gamma'$ fixes points inside of $B_{\infty}$. However, $\gamma'(B_{\infty}) \cap B_{\infty}=\emptyset$ and
$\gamma'$ does not fix $\infty$,
which is a contradiction.
\end{proof}

We are now ready to prove the following lemma.

\begin{lem}\label{noLen6exceptionalslope}
\begin{enumerate}
\item Any orbifold $Q$ with a $S^2(3,3,3)$ cusp and cusp volume $\frac{\sqrt{3}}{4}$ such that $\Gamma_Q$ 
admits integral traces cannot be covered by a knot complement. 
\item Any orbifold $Q$ with a $S^2(2,3,6)$ cusp and cusp 
volume $\frac{\sqrt{3}}{8}$ such that $\Gamma_Q$ admits integral traces cannot be covered by a knot complement. 
\end{enumerate}
\end{lem}

\begin{proof}
We begin by assuming that $Q$ has a $S^2(3,3,3)$ cusp and has cusp volume $\frac{\sqrt{3}}{4}$, and $\Gamma_Q$ admits integral traces.
Consider a horoball diagram for the fundamental domain of Q viewed from the point at $\infty$ (see Fig \ref{fig:sqrt3over4case}). 
In this figure, $O=0$, $Y=\frac{\sqrt{3}+i}{2}$, 
$X=\frac{\sqrt{3}-i}{2}$, and $D=\sqrt{3}$ and there are horoballs that are tangent to the horoball at $\infty$, which are of Euclidean diameter 1 tangent and 
to $\partial \Hth$ at $0$ and $D$. Also, there are elliptic elements of order 3 in $\Gamma_Q$ fixing 
$0$ and $\infty$, $X$ and $\infty$, $Y$ and $\infty$, and $D$ and $\infty$.

\begin{figure}
\centering
\resizebox{8cm}{!}{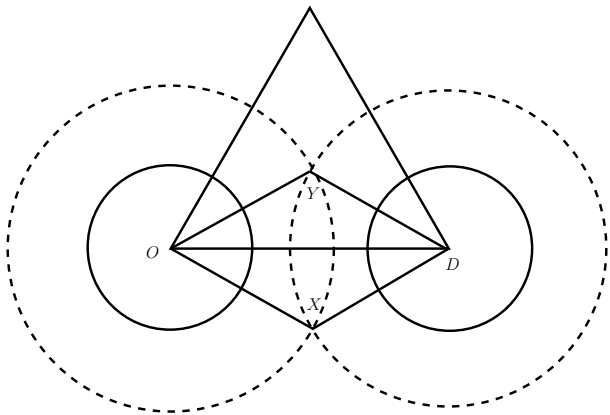}
\caption{\label{fig:sqrt3over4case} A fundamental domain for $Q$ lies above $OXDY$}
\end{figure}

We know that the point stabilizer of $0+j$ is $D_3$ or $C_3$ (see Prop \ref{propD3orC3}).

\textbf{Case 1:} Assume the point stabilizer of $0+j$ is $D_3$.

Then, there is an element $\gamma =\begin{pmatrix}
 0 & i e^{i\theta}\\
i e^{-i \theta} & 0
\end{pmatrix}$, which fixes $e^{i \theta}$ and $-e^{i \theta}$.
Therefore, the isometric sphere corresponding to $\gamma$ has radius 1 and is centered
at 0. Also, let $r$ be the element of order 3 fixing $Y$ and $\infty$. Then, $r = 
 \begin{pmatrix}
 \omega & \frac{\sqrt{3}+i}{2}\\
0 & \omega^{-1}
\end{pmatrix}$ and  $r \gamma r^{-1}$ admits an isometric sphere of radius 1 centered at 
$\sqrt{3}$. The boundaries  of these isometric spheres in $\mathbb{C}$ are depicted by dotted lines in
Figure \ref{fig:sqrt3over4case}. 
Finally, let $t = 
\begin{pmatrix}
 1 & \sqrt{3}\\
0 & 1
\end{pmatrix}.$

Since $\Gamma' = \langle \gamma, r, t \rangle$ is a subgroup with covolume $\leq v_0$ (see Fig \ref{fig:sqrt3over4case}),
it must be finite index in $\Gamma_Q$. Also, by combining the assumption that cusp volume is $\frac{\sqrt{3}}{4}$ with the upper bound on the
cusp density of $\frac{2}{v_0\sqrt{3}}$
(see $\S$\ref{sect:rigidCusps}), we know that $covolume(\Gamma_Q) \geq \frac{v_0}{2}$. Hence, $[
\Gamma_Q:\Gamma']=1,2$.

If $[\Gamma_Q:\Gamma']=2$, then $covolume(\Gamma_Q)=\frac{v_0}{2}$ and there are horoballs 
based at $\frac{\sqrt{3}+i}{2}$ and  $\frac{\sqrt{3}-i}{2}$ of Euclidean diameter 1. Thus, by Proposition \ref{propD3orC3}, the point stabilizers
above these points are either both $D_3$ or both $C_3$. Hence, the cusp corresponds to a vertex in the isotropy graph that is either 1) connected
to three vertices labeled by $D_3$ isotropy groups or  2) there is a loop labeled by 3-torsion and the cusp connects to one vertex labeled 
by a $D_3$ isotropy group (see Fig \ref{fig:loopOr3D3}). 
 In either case, $\Gamma_Q$ cannot be trivial under the cusp killing homomorphism.

\begin{figure}
\centering
\resizebox{8cm}{!}{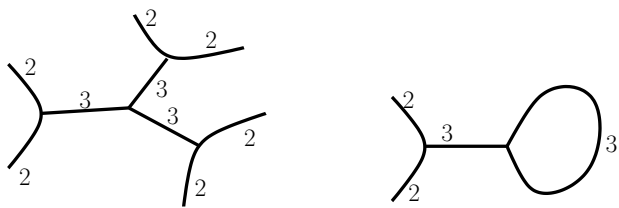}
\caption{\label{fig:loopOr3D3} The two possible types of isotropy graphs for the orbifold $Q$ described in Case 1 of Lemma \ref{noLen6exceptionalslope}
}
\end{figure}

Therefore, we consider the case that $[\Gamma_Q:\Gamma']=1$. Here, $\mbox{tr}(\gamma\cdot r) = -i (\frac{\sqrt{3}+i}{2}) e^{-i\theta}$.
Since, $-i(\frac{\sqrt{3}+i}{2})$ is a unit and we are assuming integral traces, $e^{-i \theta}$ is an algebraic integer. Hence,
$$\left\langle 
\begin{pmatrix}
 0 & i e^{i\theta}\\
i e^{-i \theta} & 0
\end{pmatrix},
\begin{pmatrix}
 1 & \sqrt{3}\\
0 & 1
\end{pmatrix},
\begin{pmatrix}
 \omega & \frac{\sqrt{3}+i}{2}\\
0 & \omega^{-1}
\end{pmatrix} \right\rangle$$ 
is a representation of $\Gamma_Q$ with parabolic elements fixing $0$ and $\infty$ where
all entries of the generators are algebraic integers. 
If $\Gamma_K \subset \Gamma_Q$, then $\GK$ admits integral representation.
 However, the maximal abelian subgroup $A_Q$ of $P_Q$ is of the form:
$$ 
\left\langle
\begin{pmatrix}
 1 & \sqrt{3}\\
0 & 1
\end{pmatrix},
\begin{pmatrix}
 1 & \omega \sqrt{3}\\
0 & 1
\end{pmatrix}
\right\rangle.$$ 
In particular, $A_Q$ vanishes under reduction modulo the prime ideal $I$, if $\sqrt{3} \in I$. 
Therefore, no knot group $\Gamma_K$ is a subgroup of $\Gamma_Q$ by Lemma \ref{lem:xIsAUnit}.

\textbf{Case 2:} The point stabilizer of $0+j$ is $C_3$. 

In this case, there is a group element $\gamma'$ that identifies $0+j$ with a point above either $\frac{\sqrt{3}+i}{2}$ or $\frac{\sqrt{3}-i}{2}$.
We may assume that $0+j$ is identified with $\frac{\sqrt{3}+i}{2} + cj$. Since $\gamma'$ can be decomposed into reflections in the plane
defined by hemisphere of
radius 1 centered at 0 

and vertical planes, $0+j$ and  $\frac{\sqrt{3}+i}{2} + cj$ are the same Euclidean distance above $\mathbb{C}$ and $c=1$. 

Hence, under $\gamma'$, $\infty \mapsto 0$, $\frac{\sqrt{3}+i}{2} \mapsto \infty$, and $\frac{\sqrt{3}+i}{2} + j \mapsto j$. Here,
$\gamma' = 
\begin{pmatrix} 
 0 & \frac{ \sqrt{3}+i}{2}\\
\frac{- \sqrt{3}+i}{2} & 1
\end{pmatrix}$. 

Let $r' =\begin{pmatrix} 
 \omega & \frac{-\sqrt{3}-3i}{2}\\
0 & \omega^{-1}
\end{pmatrix}.$  Since $\gamma'$ admits an isometric sphere of radius 1 at $\frac{ \sqrt{3}+i}{2}$, 
$\gamma'^{-1}$ admits an isometric sphere of radius 1 at $0$, and $r' \cdot \gamma'^{-1} \cdot r'^{-1}$
admits an isometric sphere of radius 1 at $\sqrt{3}$.

Hence, $\Gamma_3 = \langle \gamma', r', t \rangle$ is a subgroup of finite covolume (here $t$ is defined in Case 1) and 
$\Gamma_3$ is finite index in $\Gamma_Q$.
However, $k\Gamma_3 = \Q(\sqrt{-3})$ and $\Gamma_3$ has integral traces. Thus, $\Gamma_3$ is arithmetic and therefore, $\Gamma_Q$ is arithmetic (see \cite[Thm 8.3.2]{MR}). 
However, 
the only knot complement that can cover $Q$ is the figure 8 knot complement (see \cite{Reid1}).
Cusp volume considerations would force the figure 8 knot complement to be a 4-fold cover of $Q$.  However, $Q$ has 3 torsion on the cusp. 
Hence in this case, $Q$ is not covered
by a hyperbolic knot complement.

Finally, if $Q$ has a $S^2(2,3,6)$ cusp and cusp volume $\frac{\sqrt{3}}{8}$ and $\Gamma_Q$ admits integral traces. Then, the point stabilizer of $0+j$ is $D_6$. 
Hence, an identical argument to Case 1
shows $Q$ is not covered by a knot complement. 
\end{proof}

\subsection{Bounding the degree of covering}\label{sect:degreeBound}

In this section, we expand upon Lemma \ref{lem:lowerOrbLensBound}. However, unlike that lemma, we assume that $p\co M \rightarrow Q$ where M is a manifold.  

\begin{lem}
\begin{enumerate}\label{lem:degBound}
\item  If  $Q$ has a $S^2(3,3,3)$ cusp, then $deg(p)=12n$ $n\geq 1$.
\item If $Q$ has a $S^2(2,4,4)$ cusp, then the 
$deg(p)\geq 24$. 

\end{enumerate}
\end{lem}

\begin{proof}
\textbf{1) Assume that $Q$ has a $S^2(3,3,3)$ cusp.}
First, consider the isotropy graph of $Q$. If there is a loop in the isotropy graph,
then the other edge emanating from the cusp cannot connect the cusp to a point with isotropy group $D_3$. For this case, 
$\Gamma_Q$ would be non-trivial under the cusp killing homomorphism (see Fig \ref{fig:dIs12}). 

\begin{figure}
\centering
\resizebox{5cm}{!}{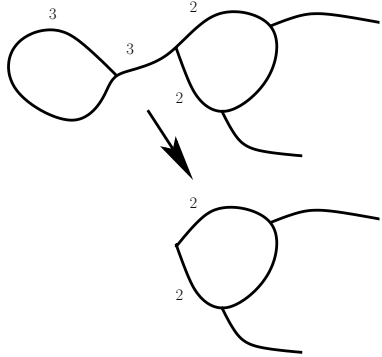}
\caption{\label{fig:dIs12} Application of the cusp killing homomorphism to a graph with a cycle labeled 3 and a vertex labeled $D_3$}
\end{figure}

Therefore, this vertex is fixed by a group $G$ where $G$ is either $A_4$, $S_4$ or $A_5$ and to lift to a torsion-free 
group $deg(p)$ must be a multiple of the order of $G$. Hence, $deg(p)=12n$ ($n \in \Z$). 

If we assume that there is no edge in the isotropy graph with both endpoints on the cusp, then 
there must be at least one vertex adjacent to the cusp labeled with a $A_4$, $S_4$ or $A_5$ subgroup. Otherwise,
all vertices are labeled with $D_3$ and just as above $\Gamma_Q$ would be non-trivial under the cusp killing homomorphism.
Thus,  $deg(p)=12n$. 

\textbf{2)  Assume that $Q$ has a  $S^2(2,4,4)$ cusp.}
Since $\Sth-K$ is a manifold, all isotropy subgroups of $\Gamma_Q$ must vanish in the
lift. Either, the order 4 elements in the cusp are part of the same fixed axis (see Fig
\ref{fig:GraphWith4Loop}) or the four torsion on the cusp connects to a pair of distinct vertices
in the isotropy graph.

In the first case, $deg(p)\geq 24$ (see Prop \ref{prop:nofourloop}). In the second case,
the vertices are either of type  $D_4$
or $S_4$ isotropy subgroup. If there is a vertex of type $S_4$, then $deg(p)\geq 24$. 
If we have a vertex of type $D_4$. There must be some edges in the isotropy graph labeled
with odd integers otherwise the graph would be non-trivial under the cusp killing homomorphism. 
Thus, $deg(p) = 8(2k+1)n$ for some $n,k\geq 1$ and  $deg(p)\geq 24$ .   
\end{proof}

We discuss length
for peripheral elements fixing $\infty$ below. Let $Q$ be a 1-cusped hyperbolic
3-orbifold and fix a representation for $\Gamma_Q$ in $\PSLTC$ such that $P_Q$ is upper triangular and 
we consider $\Gamma_Q$ acting on upper half space.
Denote by $\frac{1}{c}$ the height of a 
maximal horoball tangent to $\infty$ and denote by $S_c$ the horosphere centered at $\infty$ of Euclidean height $\frac{1}{c}$.   
If $\gamma\in \Gamma_Q$ is a parabolic element fixing $\infty$, we measure the $len(\gamma)$ by its translation length in $S_c$.
If $\gamma=\begin{pmatrix} 1 & x \\ 0 & 1 \end{pmatrix} $, then $len(\gamma)=c\cdot|x|$. 
Therefore, if $\gamma$ corresponds to an exceptional slope,
then by the Six Theorem $c \cdot |x| \leq 6$ (see \cite{Agol6Theorem}, \cite{Lackenby6Theorem}). 
Finally, by Lemma \ref{lem:xIsAUnit}, we will only consider representation of $\Gamma_Q$ such that
$\begin{pmatrix} 1 & 1 \\ 0 & 1 \end{pmatrix}\in \Gamma_Q$.  Since the interiors of maximal horoballs are disjoint, we know that $c\geq 1$. 
Also, $\begin{pmatrix} 1 & x \\ 0 & 1 \end{pmatrix}$ and $\begin{pmatrix} 1 & -x \\ 0 & 1 \end{pmatrix}$ correspond to the same
slopes in terms of Dehn surgery parameters, so for convenience we consider them as one curve in our accounting of short parabolic
elements. 

With $\Gamma_Q$, $P_Q$ as above, recall that if $A$ is the area of the the fundamental domain
for $P_Q$ in the horosphere of Euclidean height 1, the cusp volume of $Q$ is
$$\int_{\frac{1}{c}}  ^{\infty} \frac{A}{z}dz=\frac{c^2\cdot A}{2}.$$

\begin{proof}[Proof of Theorem \ref{mainthm2}.]

Assume $M$ admits two exceptional surgeries, $M$ is not covered by the figure 8 knot complement, and $M$ is covered by a small knot complement.  The second assumption is equivalent to $M$ being non-arithmetic (see \cite{Reid1}) and the third hypothesis implies that
$\GM$ has integral traces (see \cite{Bass1980}).
Also, assume $p\co M \rightarrow Q$, where $Q$ has a rigid cusp (see $\S$\ref{sect:rigidCusps}). Finally, we assume that $\GQ \subset \PSLTA$  and $P_Q$ is upper triangular (see Lem \ref{lem:intTracesUpperTri}).
We break the proof into three cases, one for each cusp type of $Q$.

\textbf{Case 1: $Q$ has a $S^2(2,4,4)$ cusp.}
We know $deg(p)\geq 24$ (see Lemma \ref{lem:degBound}).
By Lemma \ref{lem:xIsAUnit} and the Six Theorem (see \cite{Agol6Theorem}, \cite{Lackenby6Theorem}), $P_M$ is of the form:
$\left\langle \begin{pmatrix} 1 & 1 \\ 0 & 1 \end{pmatrix}, \begin{pmatrix} 1 & 6i \\ 0 & 1 \end{pmatrix} \right\rangle$ 
and $c=1$ (with $c$ defined above). In this case, the cusp volume of $Q$ is $\frac{3}{24}$.
 There is one such orbifold, which is arithmetic, contradicting our hypothesis.

\textbf{Case 2: $Q$ has a $S^2(3,3,3)$ cusp.} By Lemma \ref{lem:xIsAUnit}, we can find a representation for $\GM$ where the $P_M$ is of the form:
$
\left\langle \begin{pmatrix} 1 & 1 \\ 0 & 1 \end{pmatrix}, \begin{pmatrix} 1 & n\omega \\ 0 & 1 \end{pmatrix}\right\rangle$, 
such that $\omega^2 + \omega + 1 = 0$. By the Six Theorem (see \cite{Agol6Theorem}, \cite{Lackenby6Theorem}),
$n \leq 6$. However, $3n$ must be a multiple of 12. Hence, $n=4$ and $deg(p)=12$ (see Lemma \ref{lem:degBound}).
Here, the two shortest parabolic elements (excluding inverses) are $\mu=\begin{pmatrix} 1 & 1 \\ 0 & 1 \end{pmatrix}$ and $\lambda=\begin{pmatrix} 1 & 2+4\omega \\ 0 & 1 \end{pmatrix}$,
($|2+4\omega|=2\sqrt{3}$). 

In order to have 
two curves $\gamma_1$, $\gamma_2 \in P_M$ with $len(\gamma_i) \leq 6$, the horoballs tangent to $B_{\infty}$
have Euclidean height greater than $\frac{1}{\sqrt{3}}$ and so $c\leq \sqrt{3}$. (Note if $\lambda$ is a longer element, say $\lambda=\begin{pmatrix} 1 & 3+4\omega \\ 0 & 1 \end{pmatrix}$, then we must have $c\leq \frac{6}{\sqrt{13}}< \sqrt{3}$).
 Thus, the cusp volume of $M$ is in the range $[\sqrt{3}, 3\sqrt{3}]$
(1 $\leq |c| \leq \sqrt{3}$)
and the cusp volume of $Q$ is in the range $[\frac{\sqrt{3}}{12}, \frac{\sqrt{3}}{4}]$.

Since any orbifold $Q$ with $S^2(3,3,3)$ cusp and cusp volume 
$\frac{\sqrt{3}}{12}$, $\frac{\sqrt{3}}{6}$, or $\frac{1}{4}$ 
 is 
arithmetic (see Prop \ref{prop:AdamsAndNR2}).
Since these orbifolds are excluded by hypothesis, 
we only have to consider orbifolds with possible cusp volume,
$\frac{3\sqrt{3}+ \sqrt{15}}{24}$, $\frac{\sqrt{21}}{12}$, or $\frac{\sqrt{3}}{4}$ (see Thm \ref{thm:AdamsCusps}).  The first case, implies that $Q$ is $\Hth/\Gamma(5,2,2,3,3,3)$ which has 
an order 60 isotropy group (\cite[$\S$4.7]{MR}). Thus, it cannot have a 12-fold manifold cover. In the second case, we know no such orbifold can be covered by a knot complement by Lemma \ref{nosqrt21over12}. Finally,
we show
the third case cannot occur by appealing to Lemma \ref{noLen6exceptionalslope}.

\textbf{Case 3: $Q$ has $S^2(2,3,6)$ and $M$ admits a non-trivial symmetry} 
Since the element of $6$ torsion is part of a dihedral group of order 12, we know $deg(p)=12n$. 

If $deg(p)>24$, we claim $M$ cannot admit two exceptional surgeries. In  this case, 
$P_M =\left\langle \begin{pmatrix} 1 & 1 \\ 0 & 1 \end{pmatrix}, \begin{pmatrix} 1 & 2 n \omega \\ 0 & 1 \end{pmatrix} \right\rangle$
with
 $n > 3$ and
$\omega^2+\omega+1=0$. Since $|c|\geq 1$, if $\gamma\in P_M$ with $len(\gamma)\leq 6$, then $\gamma=\begin{pmatrix} 1 & \pm1 \\ 0 & 1 \end{pmatrix}$ 
However, these curves both correspond to surgery along the meridian.

If $deg(p)=24$, then $P_M$ is same as in Case 2 above. 
Therefore, $M$ has cusp volume in $[\sqrt{3},3\sqrt{3}]$ and $Q$ has cusp volume in $[\frac{\sqrt{3}}{24}, \frac{\sqrt{3}}{8}]$. 
For cusp volume in $[\frac{\sqrt{3}}{24}, \frac{\sqrt{3}}{8})$, 
these orbifolds fit Adams' list and only the figure 8 knot complement can 24-fold cover $Q$. This follows from Case 2 above. 
If the cusp volume is exactly $ \frac{\sqrt{3}}{8}$, we appeal to Lemma \ref{noLen6exceptionalslope}.

If $deg(p)=12$, we may consider  $ \Gamma_Q=P_Q \cdot \GM$. In this case, 
$$P_Q = \left\langle r= \begin{pmatrix} \ell & 0\\ 0 & \ell^{-1} \end{pmatrix},  t= \begin{pmatrix} 1 &  \omega \\ 0 & 1 \end{pmatrix}, \mu = \begin{pmatrix} 1 &  1 \\ 0 & 1 \end{pmatrix} \right\rangle$$
where $\ell = e^{\frac{i\pi}{6}}$ and $\mu$ is the meridian of the knot complement which covers $M$. 

Since $t\not\in\GM$, we first note that $r,r^2 \not\in N(\GM)$, where $N(
\GM)$ is the normalizer of $\GM$ in $\PSLTC$.
Also, $t \not\in N(\GM)$ (see Lem \ref{lem:lowerOrbLensBound}). 

Therefore, the only symmetry $M$ may admit can be realized by $r^3$ and we may assume $N(\GM)= \langle r^3, \GM \rangle$. Then the conjugates of $\GM$ in $P_Q\cdot \GM$ are $\GM$, $r\cdot \GM \cdot r^{-1}$,$r^{2}\cdot \GM \cdot r^{-2}$,
$t\cdot\GM\cdot t^{-1}$, $rt\cdot \GM \cdot (rt)^{-1}$ and $r^{2}t\cdot \GM \cdot t^{-1}r^{-2}$.  In this case, $t$ maps to product of three 2-cycles in $S_6$. 
Hence, $P_Q \cdot\GM$ has a $\Z/2\Z$ quotient. Therefore, $Q$ is covered by an orbifold $Q'$ with a $S^2(3,3,3)$ cusp with $[P_Q:P_{Q'}]=2$ and $\Gamma_{Q'}=P_{Q'}\cdot\GM$.
However, $M$ would be a 6-fold cover of $Q'$, where $Q'$ has a $S^2(3,3,3)$ cusp. 
However, this is a contraction to 
the minimum degree cover of $p'\co M \rightarrow Q'$ (see
Lem \ref{lem:degBound}).
\end{proof}

\section{Further Remarks}\label{sect:FurtherRemarks}

We conclude by noting that there are knot complements that admit exceptional surgeries and are not small. 
In fact, by work of Baker (see \cite{BakerLarge}), there are knot complements that admit (non-trivial) finite cyclic fillings that are not small. 
It remains unknown whether any of these knot complements admit hidden symmetries. 
 
In the case where the knot complements cover a manifold of small volume more can be said.
Consider manifolds $\beta_{n,m}$ that arise from (n,m) surgery on the unknotted cusp of 
 the Berge Manifold as in \cite{H2010}. Note, here we also impose that $(n,7)=1$, so that 
 $\beta_{n,m}$ is covered by three knot complements (see \cite[Lem 3.1]{H2010}) and $(n,m)=1$ so that  $\beta_{n,m}$ 
 is a manifold. We note that 
 $vol(\beta_{n,m}) < 4 v_0$, where $v_0$ is the volume of a the regular ideal tetrahedron.

Suppose $Q$ is an orbifold with a $S^2(3,3,3)$ or $S^2(2,4,4)$ cusp that is covered by a $\beta_{n,m}$ for some pair $(n,m)$.
 By Lemma \ref{lem:degBound}, $vol(Q)<\frac{v_0}{3}$ in the $S^2(3,3,3)$ case and $vol(Q)<\frac{v_0}{6}$ in the $S^2(2,4,4)$ case.
 In either case, such a $Q$ would be arithmetic (see Prop \ref{prop:AdamsAndNR2}) 
 and therefore could not be covered by more than one knot complement by \cite{Reid1}.
 
 If $Q$ has a $S^2(2,3,6)$ cusp,  then we use the fact that $Q$ must have an isotropy group which is dihedral and order 12. Hence,
 if $p\co \beta_{n,m} \rightarrow Q$, $deg(p)=12n$. If $deg(p)\geq24$, then $vol(Q)<\frac{v_0}{6}$ and  would be arithmetic by Proposition 
 \ref{prop:AdamsAndNR2} and just as above this is a contradiction as these orbifolds are covered by more than one knot complement (see \cite[Lem 3.1]{H2010}) and are therefore non-arithmetic (see \cite{Reid1}). If $deg(p)=12$, we note that  since the Berge manifold is strongly invertible, each $\beta_{n,m}$  can be realized 
 as the double brach cover of a tangle filling on the quotient of the Berge manifold by a 
 strong inversion. Hence, $\beta_{n,m}$ admits a symmetry. Therefore we may appeal to the proof of Theorem \ref{mainthm2} to see that 
 there can be no such covering map $p$ of degree 12. Thus, no manifold $\beta_{n,m}$ can cover an orbifold with a rigid cusp and each of the knot complements that cover a $\beta_{n,m}$ (where $\beta_{n,m}$ is a manifold) do not admit hidden symmetries (see \cite[Prop 9.1]{NR1}). 

The above argument is an improvement over \cite{H2010} 
     in the sense that the above argument  produces an explicit infinite family of knot complements that cannot admit hidden symmetries.

\bibliographystyle{plain}  
\bibliography{NRHBib} 

\end{document}